\documentclass[11pt]{article}

\usepackage{arxiv}
\usepackage[utf8]{inputenc}
\usepackage[T1]{fontenc}
\usepackage[pagebackref=true]{hyperref}
\usepackage{url}
\usepackage{booktabs}
\usepackage{amsfonts}
\usepackage{nicefrac}
\usepackage{microtype}
\usepackage{graphicx}
\usepackage{natbib}
\usepackage{doi}
\usepackage{dsfont}
\usepackage{amssymb,amsmath,amsfonts,amsthm}
\usepackage{verbatim}
\usepackage{fancyvrb}
\usepackage{bbm}
\usepackage{paralist}
\usepackage{accents}
\usepackage{verbatim}
\usepackage[ruled]{algorithm2e}
\usepackage{cleveref}
\usepackage{color}
\usepackage{caption}
\usepackage{multirow}
\usepackage{stix}
\usepackage{enumitem}
\usepackage{mathtools}
\usepackage{float}
\usepackage{multicol}
\usepackage{comment}
\usepackage{import}
\usepackage{subfig}
\usepackage{tikz-cd}
\usetikzlibrary{babel}
\usepackage{accents}
\usepackage{wrapfig}
\usepackage{extarrows}

\def\R{\mathbb{R}}
\def\N{\mathbb{N}}

\def\Q{\mathbb{Q}}
\newcommand{\E}[1]{\mathbb{E}\left[{#1}\right]}

\newcommand{\V}[1]{\mathbb{V}\left({#1}\right)}
\newcommand{\Cov}[2]{\mathrm{Cov}\left({#1},{#2}\right)}
\def\P{\mathbb{P}}
\def\X{\mathbb{X}}

\def\K{\mathcal{K}}

\def\A{\mathcal{A}}
\def\W{\mathcal{W}}

\def\RR{\mathcal{R}}
\def\MM{\mathbb{M}}

\def\XX{\mathcal{X}}

\def\QQ{\mathcal{Q}}

\def\AA{\mathbb{A}}

\DeclareMathOperator*{\esssup}{ess\,sup}

\newcommand{\1}{\mathds{1}}
\newcommand{\0}{\mathbf{0}}

\newcommand{\SKP}[2]{\langle #1 , #2 \rangle}
\newcommand{\Bf}{\mathcal{B}_f(\X)}
\newcommand{\ND}{\mathds{N}(\X)}
\newcommand{\NDM}{\mathds{N}(\X\times\MM)}
\newcommand{\NDW}{\mathds{N}(W\times\MM)}
\newcommand{\HK}[2]{{#1}^{({#2})}}
\renewcommand{\d}[1]{ \mathrm{d}{#1} }
\newcommand{\z}[3]{I_{{#1},{#2}}^{#3}}

\newcommand{\f}[1]{f_{#1}(\Delta_W)}

\renewcommand{\|}{\mid}

\newtheorem{theorem}{Theorem}[section]
\newtheorem{proposition}[theorem]{Proposition}
\newtheorem{lemma}[theorem]{Lemma}
\newtheorem{corollary}[theorem]{Corollary}
\newtheorem{remark}[theorem]{Remark}
\newtheorem{definition}[theorem]{Definition}

\hypersetup{pdftoolbar=true,
	pdfmenubar=true,
	pdfpagemode=UseOutlines,
	bookmarksnumbered=true,
	linktocpage=true,
	colorlinks=true,
	citecolor=blue,
	linkcolor=blue,
	urlcolor=blue
}

\title{Central limit theorems for the Euler characteristic in the Random Connection Model for higher-dimensional simplicial complexes}

\author{ Dominik~Pabst \\
	Institute of Stochastics\\
	Karlsruhe Institute of Technology\\
	Institute of Theoretical Physics\\
    Friedrich Alexander University Erlangen-Nuremberg
}

\renewcommand{\headeright}{}
\renewcommand{\undertitle}{}
\renewcommand{\shorttitle}{Euler characteristic in the RCM for higher-dimensional simplicial complexes}

\hypersetup{
pdftitle={Central limit theorems for the Euler characteristic in the Random Connection Model for higher-dimensional simplicial complexes},
pdfauthor={Dominik~Pabst},
pdfkeywords={Random Connection Model, Euler characteristic, Central limit theorems, Simplicial complexes},
}

\begin{document}

\maketitle

\begin{abstract}
As generalizations of random graphs, random simplicial complexes have attracted growing attention in the literature.
In this paper, we introduce a new random simplicial complex that extends the Random Connection Model (RCM), a random graph model that has been extensively studied for over three decades, to higher-dimensional simplicial complexes.
The resulting model allows simplices of different dimensions to be governed by separate connection functions, providing a higher-dimensional analogue of the RCM.
For this model, we derive explicit moment formulas for a generalized Euler characteristic and establish quantitative central limit theorems under increasing intensity and increasing observation windows.
To this end, we extend existing normal approximation results for Poisson functionals to a class of generalized difference operators.
These results are obtained in a general framework in which the vertices of the simplicial complex are drawn from an arbitrary Borel space.
In the stationary Euclidean setting with marks, we additionally establish a multivariate central limit theorem for simplex counts.
\end{abstract}

\begin{small}
\keywords{Random Connection Model \and Euler characteristic \and Central limit theorems \and Simplicial complexes}
\end{small}

\begin{small}
\mscs{60F05 \and 60D05 \and 60C05}
\end{small}

\section{Introduction}

Random graphs are used in a wide variety of fields to model real-world phenomena and simultaneously constitute a rich area of mathematical research.
Over the past two decades, random simplicial complexes, generalizations of random graphs, have gained increasing popularity.
Prominent examples include the \v Cech and Vietoris-Rips complexes, which are applied in topological data analysis to study spatial data but are now also extensively investigated mathematically outside this context.
In contrast to graphs, where edges represent interactions or connections between two objects, simplicial complexes allow for the description of interactions involving more than two objects, extending beyond pairwise interactions.
Formally, a simplicial complex is a family of non-empty finite subsets of a vertex set, closed under the subset relation.
An element of a simplicial complex consisting of $n+1$ elements is called an $n$-dimensional simplex, or simply an $n$-simplex.
The dimension of a finite-dimensional simplicial complex is the maximum dimension of its simplices.
In this sense, edges are one-dimensional simplices, and graphs are (at most) one-dimensional simplicial complexes.
The diversity of models for random simplicial complexes has steadily increased in recent years, often building upon well-known random graph models.
For instance, the famous Erd\H{o}s-R\'enyi graph was extended in \cite{Costa} to a very general model for a random simplicial complex.
The so-called clique complex of the Erd\H{o}s-R\'enyi graph was already studied in \cite{Kahle.Topology,Kahle.LimitTheorems}.
The \v Cech and Vietoris-Rips complexes generalize the random geometric graph (see \cite{Penrose.Graphs}), in which random vertices are selected within a metric space, and two vertices are joined by an edge if their distance does not exceed a fixed threshold.
Central limit theorems for Betti numbers of these two complexes based on Poisson or binomial processes can be found in \cite{Kahle.LimitTheorems,Yogeshwaran}, while \cite{Hirsch} provides a sharp phase transition for percolation of the Vietoris-Rips complex.

In this paper, we introduce a new model for a random simplicial complex based on the \textit{Random Connection Model} (RCM), one of the most extensively studied random graph models, introduced by Penrose in 1991 \cite{Penrose.RCM}.
The RCM provides a particularly flexible framework for random graphs: unlike models with a fixed connection probability, such as the Erd\H{o}s-R\'enyi graph, it combines a random vertex set with connection probabilities that may depend on the particular vertices.
In the RCM, the vertices of the random graph originate from a Poisson process on some space $\X$.
Given a measurable and symmetric function $\varphi:\X^2\rightarrow[0,1]$, each pair of vertices $x,y$ is independently connected with probability $\varphi(x,y)$.
Our construction transfers this flexibility to higher-dimensional interactions by allowing a separate connection function for simplices of each dimension.
Central limit theorems for various functionals of the RCM are provided in \cite{Can,Last.RCM}, while results on its percolation properties can be found in \cite{Caicedo,Dickson}.
To extend the RCM to higher-dimensional simplicial complexes, we consider a Poisson process $\Phi$ on a Borel space $\X$ (see Section \ref{Sec:Pre} for precise definitions) and fix some $\alpha\in\N$, which will be the maximal dimension of the simplicial complex.
For each $j\in\{1,\dots,\alpha\}$, we choose a measurable and symmetric function $\varphi_j:\X^{j+1}\rightarrow [0,1]$, referred to as the connection functions of the model.
The random simplicial complex $\Delta$ is then constructed as follows.
\begin{itemize}
\item[(0)] The vertex set of $\Delta$ is the set of points of $\Phi$.
\item[(1)] For each pair of points $x,y \in \Phi$, independently add the edge $\{x,y\}$ to the complex $\Delta$ with probability $\varphi_1(x,y)$.
\item[(2)] For each triple $x,y,z \in \Phi$ for which all three edges were added in step (1), independently add the triangle $\{x,y,z\}$ with probability $\varphi_2(x,y,z)$ to the complex $\Delta$.
\item[$\vdots$]
\item[($\alpha$)] For any $\alpha+1$ points $x_{i_0},\dots,x_{i_\alpha} \in \Phi$ for which all proper subsimplices were added in previous steps, independently add the simplex $\{x_{i_0},\dots,x_{i_\alpha}\}$ with probability $\varphi_\alpha(x_{i_0},\dots,x_{i_\alpha})$ to the complex $\Delta$.
\end{itemize}
The special case $\alpha=1$ recovers the classical RCM as a random graph.
To illustrate both the generality of the underlying space and the flexibility of the connection functions, we present two examples in hyperbolic space.
Consider the space $\X=\mathbb{H}^2:=\{ z\in\R^2 \,|\, \Vert z \Vert < 1\}$, the two-dimensional hyperbolic space realized by the Poincar\'e ball model, which is a metric space with the hyperbolic metric
\begin{align*}
d_{h}(x,y) \,:=\, \mathrm{arccosh}\left( 1+\frac{2\Vert x-y\Vert^2}{\left(1-\Vert x\Vert^2\right)\left(1-\Vert y\Vert^2\right)} \right).
\end{align*}
Hyperbolic lines in this model are either Euclidean straight lines through the origin or circles orthogonal to the unit sphere (intersected with $\mathbb{H}^2$).
A detailed introduction to hyperbolic geometry can be found in \cite{Ratcliffe}.
Both examples are based on a Poisson process $\Phi$ on $\mathbb{H}^2$ with intensity measure $\beta\lambda$ for some $\beta>0$, where
\begin{align*}
\lambda(\cdot) \,:=\, \frac{1}{2\pi} \int_0^{2\pi} \int_0^\infty \cosh(t) \; \1\{(t,\phi)\in \cdot\} \; \d t \; \d \phi.
\end{align*}
Here, $(t,\phi)$ denotes the hyperbolic point with hyperbolic distance $t$ to the origin and angle $\phi$ (from the Euclidean coordinates with respect to some reference direction).
To simplify the visual illustrations, we choose $\alpha=2$.
For the first example, consider the connection functions
\begin{align}\label{hyperbolisch_geometrisch}
\varphi_1(x,y) \,=\, \1 \big\{ d_h(x,y) \leq r \big\}, \qquad \varphi_2(x,y,z) \,\equiv\, p, \qquad r>0,\; p\in [0,1].
\end{align}
Here, the underlying graph is a random geometric graph with respect to the hyperbolic metric.
Figure \ref{fig:hyperbolicRCM1} shows a realization of the simplicial complex $\Delta$ with parameter choice $\beta=30,r=\frac{2}{5},p=\frac{1}{2}$, where the non-Euclidean geometry is clearly visible.
Besides the clustering of points near the boundary of the ball, the figure illustrates how hyperbolic distances near the boundary are considerably larger relative to Euclidean distances than near the center.
More generally, the higher-dimensional connection functions allow for a broad range of constructions by incorporating geometric or other properties of the simplices themselves; for instance, $\varphi_2(x,y,z)$ could depend on the hyperbolic area of the triangle determined by $x,y,z$ or the hyperbolic diameter of $\{x,y,z\}$.
\begin{figure}[t]
  \captionsetup{ labelfont = {bf}, format = plain }
  \centering
  \includegraphics[width=0.5\linewidth]{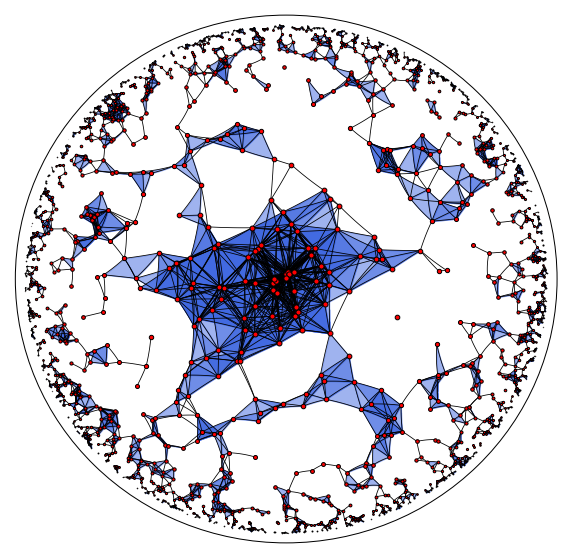}
  \caption{A realization of the model in hyperbolic space with the connection functions from (\ref{hyperbolisch_geometrisch})}
  \label{fig:hyperbolicRCM1}
\end{figure}
To illustrate a less conventional connection mechanism, we next associate points with hyperbolic lines and connect them according to intersections of the corresponding lines.
For this purpose, we choose the connection functions
\begin{align}\label{hyperbolisch_Geradenprozess}
\varphi_1(x,y) \,=\, \1 \big\{ H(x)\cap H(y) \neq \emptyset \big\}, \qquad \varphi_2(x,y,z) \,\equiv\, p, \qquad p\in [0,1],
\end{align}
where $H(z)$ is the unique hyperbolic line containing $z\in\mathbb{H}^2\setminus\{\0\}$, such that $z$ is the point of $H(z)$ with the smallest hyperbolic distance to the origin.
Therefore, we can identify the point process $\Phi$ with a line process on $\mathbb{H}^2$ (a point process on the space of hyperbolic lines).
Figure \ref{fig:hyperbolicRCM2} shows a realization of this line process and the corresponding complex $\Delta$ with respect to the connection functions from (\ref{hyperbolisch_Geradenprozess}) for $\beta=\frac{11}{5}$ and $p=\frac{1}{2}$.
\begin{figure}[t]
  \captionsetup{ labelfont = {bf}, format = plain }
  \centering
  \subfloat[][]{\includegraphics[width=0.48\linewidth]{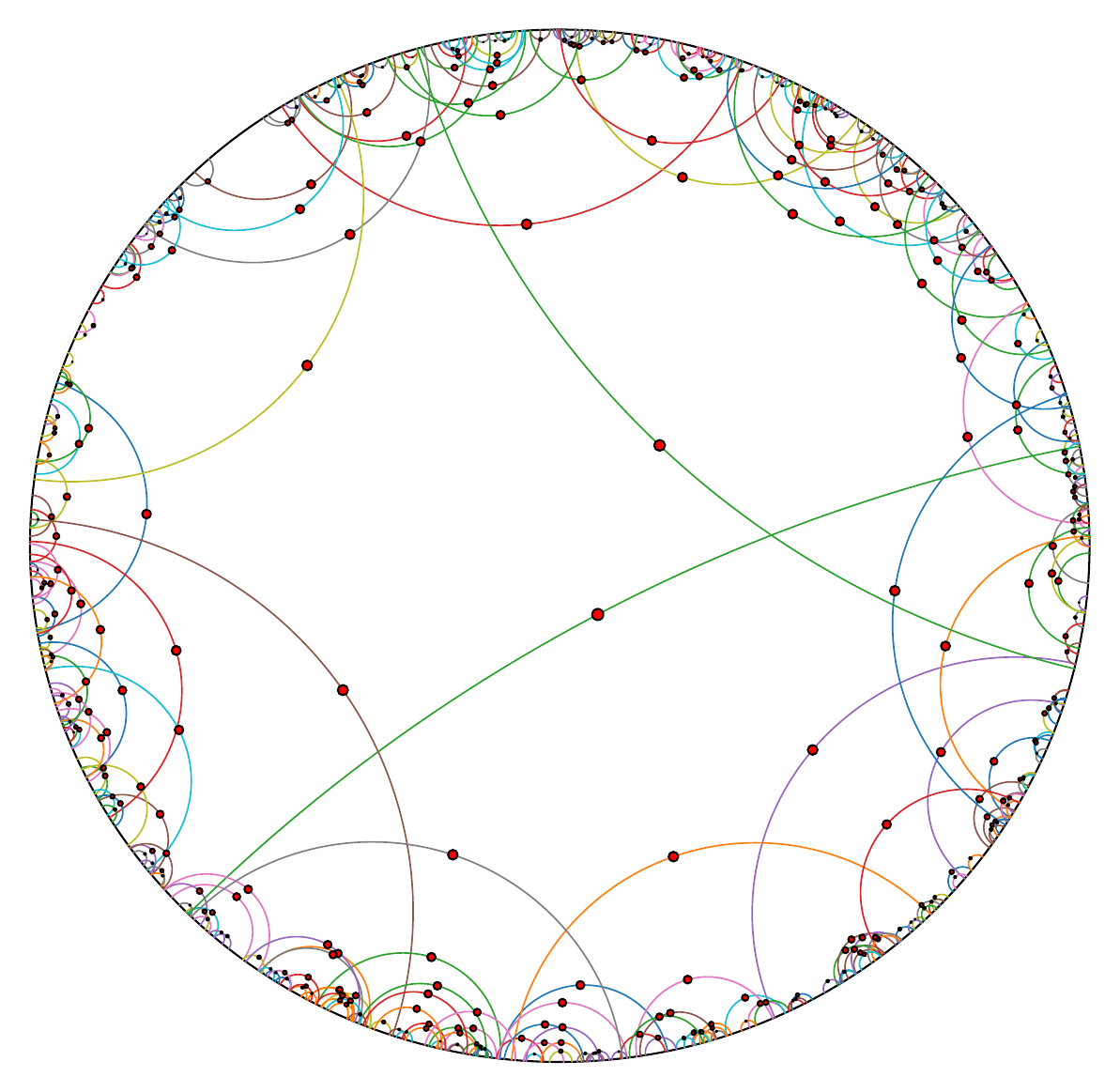}}
  \quad
  \subfloat[][]{\includegraphics[width=0.48\linewidth]{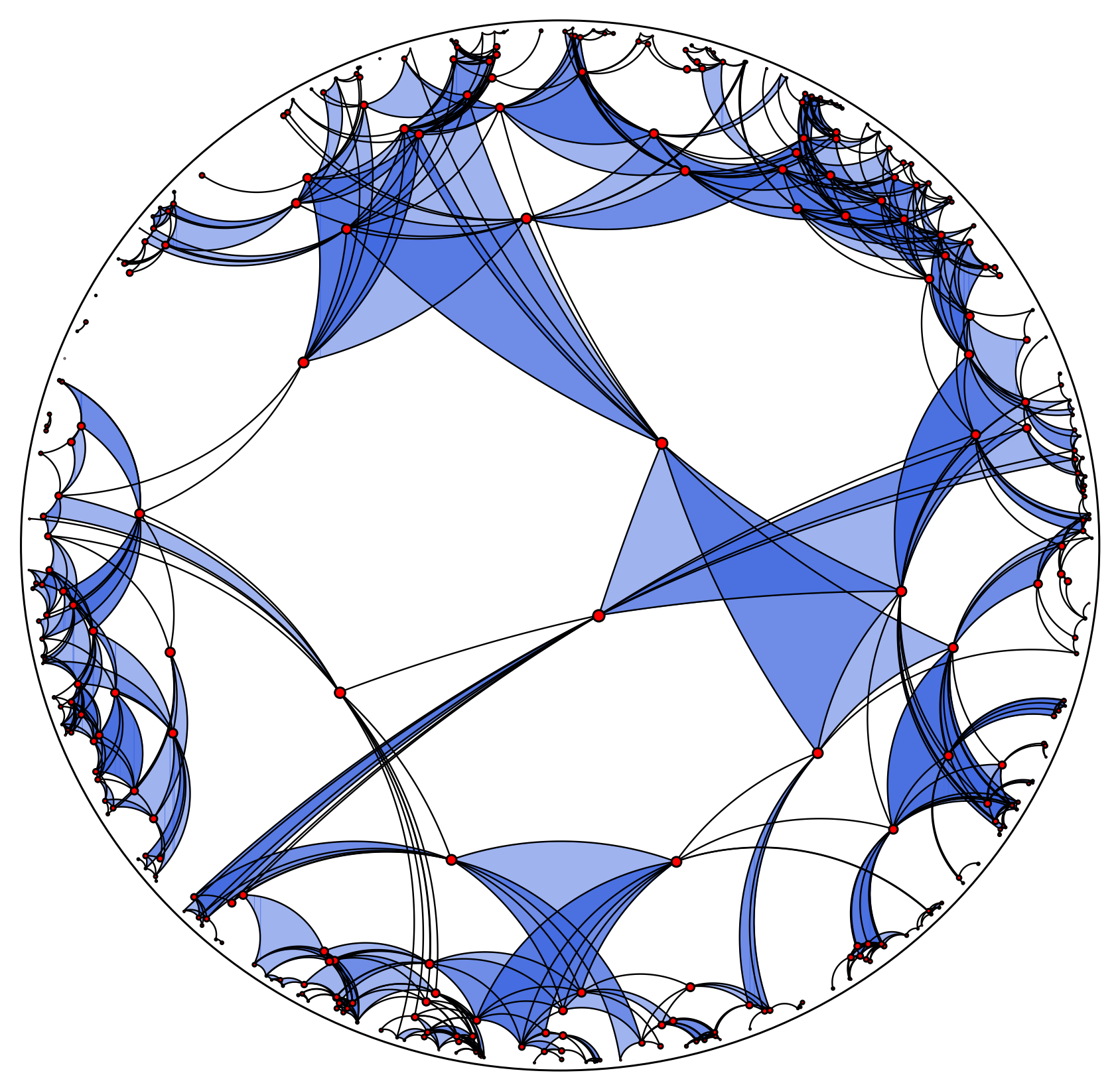}}
  \caption{A realization of a hyperbolic line process (a) and the model based on this realization (b) with the connection functions from (\ref{hyperbolisch_Geradenprozess})}
  \label{fig:hyperbolicRCM2}
\end{figure}

The $\alpha$-skeletons of the \v Cech and Vietoris-Rips complexes arise as special cases of our model, where the $\alpha$-skeleton of a simplicial complex is its restriction to simplices of dimension at most $\alpha$.
To obtain the $\alpha$-skeleton of the \v Cech complex with parameter $r>0$ on a metric space $\X$, we choose the connection functions
\begin{align*}
\varphi_j(x_0,\dots,x_j) \,=\, \1\Big\{ \bigcap_{i=0}^j B(x_i,\tfrac{r}{2})\neq\emptyset \Big\}, \qquad j\in \{1,\dots,\alpha\},
\end{align*}
where $B(x_i,s)$ denotes the closed ball with center $x_i$ and radius $s>0$ with respect to the metric on $\X$.
Selecting the connection functions
\begin{align*}
\varphi_j(x_0,\dots,x_j) \,=\, \1\big\{ \mathrm{diam}(x_0,\dots,x_j)\leq r \big\}, \qquad j\in \{1,\dots,\alpha\},
\end{align*}
yields the $\alpha$-skeleton of the Vietoris-Rips complex, where $\mathrm{diam}(x_0,\dots,x_j)$ denotes the diameter of the set $\{x_0,\dots,x_j\}$ with respect to the metric on $\X$.
The $\alpha$-skeletons of the soft random simplicial complexes from \cite{Candela} also arise as special cases of the model presented in this paper.
In the model of \cite{Costa}, the probability of adding a simplex at a given step depends on its dimension but not on the particular vertices forming the simplex.
Similarly, the soft random simplicial complexes of \cite{Candela} introduce dimension-dependent probabilities into the geometric construction of the \v Cech and Vietoris-Rips complexes.
In contrast, the connection functions $\varphi_j$ in our model may depend on the particular vertices $x_0,\dots,x_j$ and can be specified separately for each dimension $j$.
Thus, the proposed construction provides a common framework for these models while allowing substantially more general connection mechanisms.

Having introduced the model and some of its special cases, we now turn to the functional studied throughout this paper.
Together with the Betti numbers, the \textit{Euler characteristic} is one of the most widely studied functionals of simplicial complexes and is directly related to the Betti numbers through the Euler--Poincar\'e formula.
Central limit theorems for the Euler characteristic of random simplicial complexes can be found in \cite{Candela,Thomas}.
The results of the present paper extend this line of research to the broad class of random simplicial complexes introduced above.
However, the Euler characteristic is not only investigated for simplicial complexes but also in other contexts.
In stochastic geometry, the so-called intrinsic volumes are often of interest.
They can be defined on polyconvex sets (unions of finitely many compact and convex subsets of $\R^d$), and the zeroth intrinsic volume is the Euler characteristic (see for example Section 14.2 in \cite{Schneider}).
Since our methods naturally apply to arbitrary coefficient vectors, we consider a generalized Euler characteristic, which will also allow us to derive a multivariate central limit theorem for simplex counts.
For a finite simplicial complex $K$ we define the Euler characteristic with coefficient vector $a=(a_0,\dots,a_\alpha)\in\R^{\alpha+1}$ by
\begin{align*}
\chi_a(K) \,:=\, \sum_{i=0}^{\dim(K)} a_i \, f_i(K),
\end{align*}
where $f_i(K)$ is the number of $i$-simplices contained in $K$ and $a_i=0$ for $i>\alpha$.
The classical Euler characteristic is obtained by choosing $a_i=(-1)^i$.
Since the intensity measure of the underlying Poisson process $\Phi$ is not necessarily finite, $\Delta$ may contain an infinite number of vertices.
We assume that the intensity measure of $\Phi$ has the form $\beta\lambda$, where $\lambda$ is a diffuse, $\sigma$-finite measure, and $\beta > 0$.
To analyze the Euler characteristic, we often restrict $\Delta$ to simplices whose vertices lie within an observation window $W \subseteq \X$, where $W$ is a measurable set satisfying $\lambda(W) < \infty$.
This restricted subcomplex, denoted by $\Delta_W$, is almost surely finite.
The results and methods presented here are derived from the PhD thesis \cite{Pabst.Thesis}.

Our analysis of the new model begins with explicit formulas for the first and second moments of the generalized Euler characteristic $\chi_a(\Delta_W)$, derived in Section \ref{Sec:MomentFormulas}.
One of the main contributions of this paper is the establishment of quantitative central limit theorems for $\chi_a(\Delta_W)$ in two asymptotic regimes: increasing intensity, $\beta\rightarrow\infty$, and increasing observation windows, $\lambda(W)\rightarrow\infty$.
For this purpose, in Section \ref{Sec:NormalAppTheorem} we introduce generalized difference operators that incorporate additional auxiliary points, allowing different operators to be coupled by using the same augmented point configuration, and extend the normal approximation results for Poisson functionals from \cite{Last.NormalApprox} to these operators.
The resulting framework, together with additional tools developed in that section, is then used to derive the central limit theorems in Section \ref{Sec:CLT}.
Finally, in Section \ref{Sec:MultivariateCLT}, we specialize to the marked stationary case and establish a multivariate central limit theorem for simplex counts.

\section{Preliminaries} \label{Sec:Pre}

Let $(\X,\mathcal{X})$ be a Borel space, that is a measurable space, such that there is a measurable bijection from $\X$ to a Borel subset of $[0,1]$ with measurable inverse.
We equip $(\X,\mathcal{X})$ with a $\sigma$-finite and diffuse measure $\lambda$.
Throughout the paper we will often use the notation $|W|:=\lambda(W)$ for $W\in\mathcal{X}$.
Furthermore, we assume there is a transitive binary relation $\prec$ on $\X$ such that $\{(x,y)\in \X^2 \;|\; x\prec y\}$ is a measurable subset of $\X^2$ and $\lambda([x])=0$ for all $x\in\X$, where $[x]:=\X\setminus \{y\in \X \;|\; x\prec y \text{ or } y\prec x\}$.
In addition, we require that $x\prec x$ fails for all $x\in\X$.
Observe that the requirements on the relation $\prec$ already imply that the measure $\lambda$ is diffuse.
The existence of such a relation entails no loss of generality, since for an arbitrary Borel space $(\mathbb{Y},\mathcal{Y})$ equipped with a $\sigma$-finite measure $\mu$ we can always consider the extended Borel space $\X:=\mathbb{Y}\times [0,1]$ with the measure $\lambda:=\mu\otimes \lambda^1\vert_{[0,1]}$ and the relation $(x,u)\prec (y,v)$ if $u<v$, where $\lambda^1\vert_{[0,1]}$ denotes the restriction of the one-dimensional Lebesgue measure to the unit interval $[0,1]$.
We assume that all random objects appearing are defined on the same probability space, denoted by $(\Omega,\A,\P)$.

\subsection{Poisson processes}

Let $\ND$ be the set of all measures on $(\X,\mathcal{X})$ which can be written as an at most countable sum of finite measures taking values in $\N_0$ and $\mathcal{N}(\X)$ the smallest $\sigma$-field, such that the mappings $\ND\rightarrow\R$, $\eta\mapsto\eta(B)$ are measurable for all $B\in\mathcal{X}$.
A random element $\Phi$ in $(\ND,\mathcal{N}(\X))$ is called a point process on $\X$ and the measure $B\mapsto \E{\Phi(B)}$ the intensity measure of $\Phi$.
For the construction of the model we will consider a so-called Poisson process, which is a point process with certain properties.
These properties are crucial for the results presented in this paper.
For a comprehensive introduction to the theory of Poisson processes, we refer to \cite{Last.Lectures}.
A point process $\Phi$ on $\X$ with intensity measure $\beta\lambda$ for $\beta>0$ is called a Poisson process if it satisfies the following two properties.
\begin{enumerate}
    \item[(a)] For every $B\in\XX$ the random variable $\Phi(B)$ follows a Poisson distribution with parameter $\beta\lambda(B)$.
    \item[(b)] For each $m\in\N$ and pairwise disjoint sets $B_1,\dots,B_m\in\XX$, the random variables $\Phi(B_1),\dots,\Phi(B_m)$ are stochastically independent.
\end{enumerate}

A Poisson process always exists in this setting (Theorem 3.6 in \cite{Last.Lectures}) and its distribution is uniquely determined by its intensity measure (Proposition 3.2 in \cite{Last.Lectures}).
From now on let $\Phi$ be a Poisson process on $\X$ with intensity measure $\beta\lambda$.
Corollary 6.5 in \cite{Last.Lectures} shows the existence of a random variable $\tau^\prime$ with values in $\N_0\cup\{\infty\}$ and random elements $X_1^\prime,X_2^\prime,\dots$ in $\X$ such that
\begin{align} \label{PktP:Darstellung}
\Phi \,=\, \sum_{i=1}^{\tau^\prime} \delta_{X_i^\prime}
\end{align}
holds almost surely.
The restriction of $\Phi$ to a set $\W\in\XX$ is denoted by $\Phi_W$.
Observe that $\Phi_W$ is a Poisson process with intensity measure $\beta\lambda_W:=\beta\lambda(\cdot\cap W)$.
We call a set $W\in\XX$ an observation window if $0<\lambda(W)<\infty$ and denote the set of all observation windows by $\Bf$.
For $W\in\Bf$ Proposition 3.5 in \cite{Last.Lectures} shows that $\Phi_W$ satisfies
\begin{align} \label{PhiW:Darstellung}
\Phi_W \,\stackrel{d}{=}\, \sum_{i=1}^{\Phi(W)} \delta_{X_i},
\end{align}
where $X_1,X_2,\dots$ are i.i.d. random elements in $\X$ with distribution $\frac{\lambda(\cdot\cap W)}{\lambda(W)}$ independent of $\Phi(W)$.
An important construction is that of so-called marked point processes.
Given some measurable space $(\MM,\mathcal{M})$, a probability measure $\Q$ on $(\MM,\mathcal{M})$ and a sequence $(U_i)_{i\in\N}$ of i.i.d. random elements in $\MM$ with distribution $\Q$, which is independent of $\Phi$, the point process
\begin{align*}
\Psi \,:=\, \sum_{i=1}^{\tau^\prime} \delta_{(X_i^\prime,U_i)}
\end{align*}
is called an independent $\Q$-marking of $\Phi$, where we used the representation (\ref{PktP:Darstellung}) of $\Phi$.
Theorem 5.6 in \cite{Last.Lectures} shows that $\Psi$ is a Poisson process with intensity measure $\beta\lambda\otimes\Q$.

\subsection{Simplicial complexes}

An abstract simplicial complex over a vertex set $V$ is a set of nonempty finite subsets of $V$, which is closed under taking subsets.
For the vertex set $V$ of an abstract simplicial complex $K$ being uniquely determined, we will always assume that all singleton subsets of $V$ are elements of $K$.
An element $\sigma$ of a simplicial complex $K$ with $|\sigma|=j+1$ is called an abstract $j$-simplex and $\dim(\sigma):=|\sigma|-1=j$ is called its dimension.
We denote a $j$-simplex, whose elements are $v_0,\dots,v_j$, by $[v_0,\dots,v_j]$.
Since we only consider abstract simplicial complexes in this paper, we will omit the word abstract.
The dimension of a simplicial complex $K$ is defined by $\dim(K):=\sup\{\dim(\sigma) \mid \sigma\in K \}$ and can be infinite, while a one-dimensional simplicial complex is a graph.
The restriction of $K$ to simplices of dimension not bigger than some $j\in\N_0$ is called the $j$-skeleton
\begin{align*}
S_j(K) \,:=\, \{ \sigma\in K \,\|\, \dim(\sigma)\leq j \}.
\end{align*}
of $K$ and is again a simplicial complex.
Two simplicial complexes $K,L$ with vertex sets $V_K,V_L$ are called isomorphic, denoted by $K\cong L$, if there is a bijection $f:V_K\rightarrow V_L$ such that
\begin{align*}
\sigma\in K \quad \Longleftrightarrow \quad f(\sigma)\in L.
\end{align*}
Loosely speaking, two simplicial complexes are isomorphic if we can get one out of the other just by renaming the vertices.
Denoting the number of $j$-simplices of a simplicial complex $K$ by $f_j(K)$, we define the Euler characteristic with coefficient vector $a=(a_0,\dots,a_\alpha)\in\R^{\alpha+1}$ of a finite simplicial complex $K$ by
\begin{align}\label{Def:Euler}
\chi_a(K) \,:=\, \sum_{i=0}^{\dim(K)} a_i f_i(K),
\end{align}
where $a_i=0$ for $i>\alpha$.
The classical Euler characteristic arises in the special case $a_i=(-1)^i$.
Moreover, we write $F_i(K)$ for the set of $i$-simplices of $K$.
To simplify the description of concrete simplicial complexes, we introduce the following notation.
For an at most countable set $M$ of abstract simplices we denote by
\begin{align*}
  \langle M \rangle \,:=\, \{ \rho \,|\, \emptyset\neq\rho\subseteq\sigma \text{ for some } \sigma\in M \} 
\end{align*}
the smallest simplicial complex containing all the simplices in $M$.
We call $\langle M \rangle$ the simplicial complex generated by the simplices of $M$.
To simplify the notation, we omit set brackets using this notation.
For example, the simplicial complex $K$ from Figure \ref{fig:2dimSK} can be written as
\begin{align*}
K \,=\, \big\langle [x_1,x_2,x_3], [x_3,x_4], [x_3,x_5], [x_4,x_5], [x_5,x_6,x_7] \big\rangle.
\end{align*}

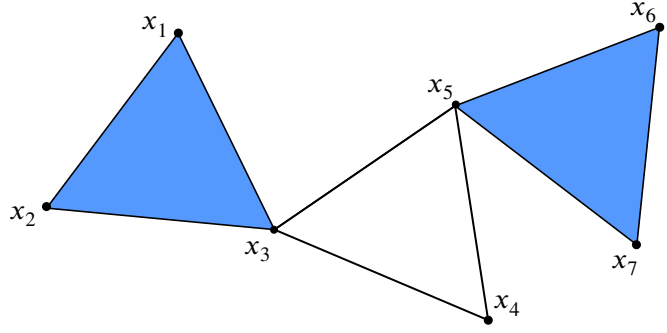
\begin{figure}[t!]
    \captionsetup{
    labelfont = {bf},
    format = plain,
  }
	\centering
	\begin{normalsize}
\begingroup%
  \makeatletter%
  \providecommand\color[2][]{%
    \errmessage{(Inkscape) Color is used for the text in Inkscape, but the package 'color.sty' is not loaded}%
    \renewcommand\color[2][]{}%
  }%
  \providecommand\transparent[1]{%
    \errmessage{(Inkscape) Transparency is used (non-zero) for the text in Inkscape, but the package 'transparent.sty' is not loaded}%
    \renewcommand\transparent[1]{}%
  }%
  \providecommand\rotatebox[2]{#2}%
  \newcommand*\fsize{\dimexpr\f@size pt\relax}%
  \newcommand*\lineheight[1]{\fontsize{\fsize}{#1\fsize}\selectfont}%
  \ifx\svgwidth\undefined%
    \setlength{\unitlength}{267.73769673bp}%
    \ifx\svgscale\undefined%
      \relax%
    \else%
      \setlength{\unitlength}{\unitlength * \real{\svgscale}}%
    \fi%
  \else%
    \setlength{\unitlength}{\svgwidth}%
  \fi%
  \global\let\svgwidth\undefined%
  \global\let\svgscale\undefined%
  \makeatother%
  \begin{picture}(1,0.47021875)%
    \lineheight{1}%
    \setlength\tabcolsep{0pt}%
    \put(0,0){\includegraphics[width=\unitlength,page=1]{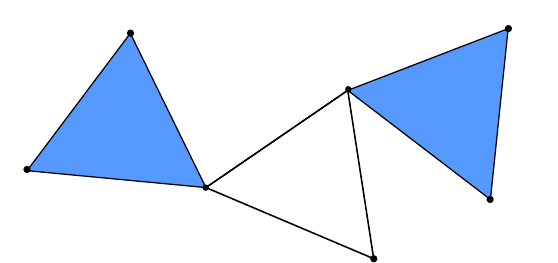}}%
    \put(0.1819526,0.41607067){\color[rgb]{0,0,0}\makebox(0,0)[lt]{\lineheight{1.25}\smash{\begin{tabular}[t]{l}$x_1$\end{tabular}}}}%
    \put(-0.00372041,0.14900006){\color[rgb]{0,0,0}\makebox(0,0)[lt]{\lineheight{1.25}\smash{\begin{tabular}[t]{l}$x_2$\end{tabular}}}}%
    \put(0.32674331,0.09997175){\color[rgb]{0,0,0}\makebox(0,0)[lt]{\lineheight{1.25}\smash{\begin{tabular}[t]{l}$x_3$\end{tabular}}}}%
    \put(0.67758246,0.02292721){\color[rgb]{0,0,0}\makebox(0,0)[lt]{\lineheight{1.25}\smash{\begin{tabular}[t]{l}$x_4$\end{tabular}}}}%
    \put(0.58398295,0.32792171){\color[rgb]{0,0,0}\makebox(0,0)[lt]{\lineheight{1.25}\smash{\begin{tabular}[t]{l}$x_5$\end{tabular}}}}%
    \put(0.86923867,0.43616609){\color[rgb]{0,0,0}\makebox(0,0)[lt]{\lineheight{1.25}\smash{\begin{tabular}[t]{l}$x_6$\end{tabular}}}}%
    \put(0.84376953,0.07959635){\color[rgb]{0,0,0}\makebox(0,0)[lt]{\lineheight{1.25}\smash{\begin{tabular}[t]{l}$x_7$\end{tabular}}}}%
  \end{picture}%
\endgroup%

	\end{normalsize}
		 
	\caption{A two-dimensional simplicial complex $K$ with vertex set $\{x_1,\dots,x_7\}$}
	\label{fig:2dimSK}
\end{figure}

\section{Definition of the model} \label{Sec:Definition_RCM}

The goal of this section is to provide a concrete definition of the random simplicial complex $\Delta$.
We fix some $\alpha\in\N$, which will be the maximal possible dimension of $\Delta$, and choose for each $j\in\{1,\dots,\alpha\}$ a measurable and symmetric function
\begin{align*}
\varphi_j: \X^{j+1} \rightarrow [0,1].
\end{align*}
The function $\varphi_j$ governs the inclusion of $j$-simplices in $\Delta$.
More specifically, the value $\varphi_j(x_0,\dots,x_j)$ is the probability that the simplex $[x_0,\dots,x_j]$ is contained in $\Delta$, conditional on all its proper subsimplices being contained in $\Delta$.
We call $\varphi_1,\dots,\varphi_\alpha$ the connection functions of the model.
For the remainder of this paper, let $\Phi$ be a Poisson process on $\X$ with intensity measure $\beta\lambda$ for some $\beta>0$.
We refer to $\beta$ as the intensity of the model and assume that $\lambda$ is $\sigma$-finite and diffuse.
Additionally, we fix the binary relation $\prec$ on $\X$ introduced in Section \ref{Sec:Pre}.
The properties of $\prec$ ensure that the points of $\Phi$ can almost surely be ordered with respect to $\prec$.

We now construct the random simplicial complex $\Delta$ from the Poisson process $\Phi$ and the connection functions $\varphi_1,\dots,\varphi_\alpha$.
First, the vertex set of $\Delta$ consists of the points of $\Phi$.
Using the representation (\ref{PktP:Darstellung}) of $\Phi$, the vertex set of $\Delta$ is given by $\{X_i^\prime\mid i\leq\tau^\prime\}$.
To determine which higher-dimensional simplices are included in $\Delta$, we require additional randomness.
For this purpose, we equip the points of $\Phi$ with additional independent marks that provide an independent uniform random variable for every potential simplex.
By comparing these random variables with the corresponding value of the connection function, we decide whether a simplex is included in $\Delta$.
We write
\begin{align*}
M^{(j)} \,:=\, [0,1]^{\N^j} \,=\, \big\{ (a_{z})_{z\in\N^j} \,\|\, a_{z}\in [0,1] \text{ for all } z\in\N^j \big\}
\end{align*}
for the space of $j$-fold indexed sequences with entries in $[0,1]$, equipped with the product $\sigma$-field of the Borel $\sigma$-fields on $[0,1]$.
Consider the product space $\MM := M^{(2)} \times \dots \times M^{(2\alpha)}$ equipped with the product $\sigma$-field and the probability distribution $\Q:=\otimes_{j=1}^\alpha\otimes_{z\in\N^{2j}}\mathcal{U}([0,1])$, where $\mathcal{U}([0,1])$ is the uniform distribution on $[0,1]$.
A random element $U$ in $\MM$ with distribution $\Q$ can be viewed as a family of i.i.d. random variables with distribution $\mathcal{U}([0,1])$, indexed by $\cup_{j=1}^\alpha\N^{2j}$.
We denote the random variable corresponding to the index $(m_1,l_1,\dots,m_j,l_j)\in\N^{2j}$ by $U_{(m_1,l_1,\dots,m_j,l_j)}$.
From now on, let $\Psi$ be the independent $\Q$-marking of $\Phi$ given by
\begin{align*}
    \Psi \,:=\, \sum_{i=1}^{\tau^\prime} \delta_{(X_i^\prime,U^{(i)})},
\end{align*}
where $(U^{(i)})_{i\in\N}$ is a sequence of i.i.d. random elements in $\MM$ with distribution $\Q$.
Thus, each point $X_i^\prime$ of $\Phi$ carries its own independent collection $U^{(i)}$ of uniform random variables.
To assign the random variables contained in the marks $U^{(i)}$, $i\in\N$, to potential simplices in a way that is independent of the enumeration of $\Phi$, we next assign a pair of indices in $\N^2$ to each point of $\Phi$.
Let $\mathcal{Q}=(Q_m)_{m\in\N}$ be a partition of $\X$ into measurable sets with $\lambda(Q_m)<\infty$ for all $m\in\N$.
Such a partition always exists due to the $\sigma$-finiteness of $\lambda$.
Recall that, almost surely, any two distinct points of $\Phi$ are comparable with respect to $\prec$.
For $i\in\N$ with $i\leq\tau^\prime$, let $M_i\in\N$ be the unique number with $X_i^\prime\in Q_{M_i}$ and $L_i\in\N$ be the unique number such that $X_i^\prime$ is the $L_i$-th smallest point of $\Phi$ in $Q_{M_i}$ with respect to $\prec$.
The latter means $\Phi( \{x\in Q_{M_i} \mid x\prec X_i^\prime \} )=L_i-1$.
Since $\Phi(Q_m)<\infty$ almost surely for every $m\in\N$ and the partition is countable, all $L_i$ are finite almost surely.
The tuple $(M_i,L_i)$ can thus be interpreted as the coordinates of $X_i^\prime$ in $\Phi$: $M_i$ identifies the partition set containing the point, while $L_i$ gives the rank of $X_i^\prime$ within this set with respect to $\prec$.

To determine whether a simplex $\sigma:=[X_{i_1}^\prime,\dots,X_{i_j}^\prime]$ with $1\leq i_1 <\dots < i_j\leq \tau^\prime$, $2\leq j\leq \alpha+1$ belongs to $\Delta$, we assign a uniformly distributed variable to $\sigma$.
Almost surely we can order the elements $X_{i_1}^\prime,\dots,X_{i_j}^\prime$ with respect to $\prec$, i.e., there is a permutation $\pi$ of $\{1,\dots,j\}$ such that $X_{i_{\pi(1)}}^\prime\prec\dots\prec X_{i_{\pi(j)}}^\prime$.
For each potential simplex, we use the mark attached to its largest vertex with respect to $\prec$ and select from this mark the uniform random variable indexed by the coordinates of the remaining vertices:
\begin{align*}
U_\sigma \,:=\, U^{(i_{\pi(j)})}_{(M_{i_{\pi(1)}},L_{i_{\pi(1)}},\dots,M_{i_{\pi(j-1)}},L_{i_{\pi(j-1)}})}.
\end{align*}
Now define
\begin{align*}
    \sigma\in\Delta \enskip :\Longleftrightarrow \enskip U_\rho\leq \varphi_{|\rho|-1}(X_\rho^\prime) \enskip \text{for all } \rho\subseteq\sigma \text{ with } |\rho|\geq 2,
\end{align*}
where $\varphi_{|\rho|-1}(X_\rho^\prime)$ denotes the value of $\varphi_{|\rho|-1}$ evaluated at the vertices of $\rho$.
This construction almost surely uniquely determines the random simplicial complex $\Delta$ as a function of $\Psi$.
In particular, requiring the above condition for every non-singleton subset $\rho$ of $\sigma$ ensures that $\Delta$ is closed under taking subsets and therefore indeed a simplicial complex.
A similar way of assigning random variables contained in point marks to the edges in the RCM is already used in \cite{Last.RCM}.
We explain some important properties of this construction.

First, the construction of $\Delta$ does not depend on the enumeration of the points of $\Phi$ in representation (\ref{PktP:Darstellung}).
Indeed, the coordinates $(M_i,L_i)$ are determined by the partition $\mathcal{Q}$, the relation $\prec$, and the point configuration $\Phi$, and hence the assignment of the random variables to potential simplices is independent of the enumeration.
For a set $W\in\XX$, we denote the restriction of $\Delta$ to simplices with vertices in $W$ by $\Delta_W$.
Observe that the coordinates $(M_i,L_i)$ of $X_i^\prime$ only depend on $\Phi_{Q_{M_i}}$.
Therefore, for every $i\in\N$, the complex $\Delta_{Q_i}$ is already determined by $\Psi_{Q_i}:=\Psi_{Q_i\times\MM}$.
More generally, for $B\in\XX$, the complex $\Delta_B$ is determined by $\Psi_{Q_B}$, where $Q_B$ is the union of all sets of the partition $\QQ$ which have a non-empty intersection with $B$.
This locality property is an important advantage of the chosen construction.

To simplify the notation in some places, we formally set $\varphi_0\equiv 1$.
With this convention, vertices can formally be treated in the same way as higher-dimensional simplices, while every point of $\Phi$ forms a vertex of $\Delta$ with probability one.
In principle, it would also be possible to choose any measurable function $\varphi_0:\X\rightarrow[0,1]$ and independently include each point $x\in\Phi$ in the vertex set with probability $\varphi_0(x)$.
However, the choice of any such function $\varphi_0$ would not be a generalization of the model, as this would merely correspond to replacing $\Phi$ by a $\varphi_0$-thinning (see Definition 5.7 in \cite{Last.Lectures}) of $\Phi$.
According to Corollary 5.9 in \cite{Last.Lectures}, such a thinning results in another Poisson process with intensity measure
\begin{align*}
(\beta\varphi_0\lambda)(A) \,:=\, \beta \int_A \varphi_0(x) \; \lambda (\d x), \qquad A \in \mathcal{X},
\end{align*}
which preserves the model framework, as the original requirements on the intensity measure (i.e., $\sigma$-finiteness and diffuseness) remain satisfied.

\section{First and second moments} \label{Sec:MomentFormulas}

In this section, we will derive formulas for the expected value and the variance of the Euler characteristic.
We will not limit ourselves to the classical Euler characteristic but take an arbitrary linear combination of simplex counts into account, i.e., from now on, we fix a coefficient vector $a=(a_0,\dots,a_\alpha)\in\R^{\alpha+1}$ and consider the generalized Euler characteristic defined in (\ref{Def:Euler}).
This will enable us to establish a multivariate central limit theorem for simplex counts in Section \ref{Sec:MultivariateCLT}.
Before deriving the moment formulas, we introduce so-called integral representations of simplicial complexes, which will play a crucial role not only in this section but throughout the paper.
In particular, we will see in Section \ref{Sec:CLT} that controlling these integral representations will be the decisive step in proving a central limit theorem for the Euler characteristic.
For the whole section, we fix an observation window $W\in\Bf$.

\begin{definition} [Simplex functions and integral representations] \label{Def:Simplexfkt}
Let $K$ be a simplicial complex with $\dim(K)\leq\alpha$ on the vertex set $\{1,\dots,r\}$ for some $r\in\N$.
\begin{enumerate}
\item Define the simplex function of $K$ by
\begin{align*}
f_K: \X^r \rightarrow [0,1], \qquad f_K(x_1,\dots,x_r):= \prod_{\sigma\in K} \; \varphi_{|\sigma|-1} (x_\sigma),
\end{align*}
where $\varphi_{|\sigma|-1} (x_\sigma)$ denotes the value of $\varphi_{|\sigma|-1}$ evaluated at the points $\{ x_i\,|\, i\in \sigma\}$.
\item The integral representation $I_K$ of $K$ (over an observation window $W$) is defined by
\begin{align*}
I_K(W) \,:=\, \int_{W^r} f_K(x_1,\dots,x_r) \;\; \lambda^r \, (\d {(x_1,\dots,x_r)}).
\end{align*}
We call $K$ the simplicial complex associated to $I_K$.
\item For $m,l\in\{1,\dots,\alpha+1\}$ and $\max\{m,l\}\leq r\leq m+l$, we consider the simplicial complex
\begin{align} \label{Def:Komplex_Kmlr}
K_{m,l}^r \,:=\, \Big\langle [1,\dots,m], [1,\dots,m+l-r,m+1,\dots,r] \Big\rangle
\end{align}
on the vertex set $\{1,\dots,r\}$, which is generated by a $(m-1)$-simplex and a $(l-1)$-simplex sharing $m+l-r$ vertices.
Moreover, we denote by $\kappa_{m,l}^r$ the simplex function and by $\z m l r$ the integral representation of $K_{m,l}^r$, which means
\begin{align*}
\z m l r (W) \,:=\, \int_{W^r} \kappa_{m,l}^r(x_1,\dots,x_r) \;\; \lambda^r \, (\d {(x_1,\dots,x_r)}).
\end{align*}
Finally, let $\kappa_m:=\kappa_{m+1,m+1}^{m+1}$ be the simplex function of a simplicial complex generated by an $m$-simplex.
\end{enumerate}
\end{definition}

The simplex function of a simplicial complex is not necessarily symmetric, nor are the simplex functions of isomorphic simplicial complexes generally identical.
However, the simplex functions of isomorphic simplicial complexes can be derived from one another by exchanging the arguments.
Consequently, integral representations of isomorphic simplicial complexes coincide.
From the definition of integral representations, it is also evident that the integral representation of a simplicial complex is the product of the integral representations of its connected components.
Conversely, a product of integral representations is again an integral representation, where the numbers of connected components are added.

For the definition of the integral representation $\z m l r$, the specific choice of the simplicial complex $K_{m,l}^r$ does not matter, only its isomorphism class.
The latter is already determined by the fact that $K$ is generated by two simplices of dimensions $m-1$ and $l-1$, which share $m+l-r$ vertices.
The functions $\kappa_m$ can also be explicitly expressed by
\begin{align*}
    \kappa_m(x_0,\dots,x_{m}) \,:=\, \prod_{\emptyset\neq I\subseteq\{0,\dots,m\}} \varphi_{|I|-1}(x_I),
\end{align*}
where $\varphi_{|I|-1}(x_I)$ denotes the value of $\varphi_{|I|-1}$ evaluated at the points $x_i$ with $i\in I$.
For $m\in\{1,2\}$, we obtain, for example,
\begin{align*}
\kappa_1 \,=\, \varphi_1,\qquad \kappa_2(x,y,z) \,=\, \varphi_2(x,y,z)\cdot\varphi_1(x,y)\cdot\varphi_1(x,z)\cdot\varphi_1(y,z).
\end{align*}
The integral representation of the simplicial complex from Figure \ref{fig:2dimSK} is given by
\begin{align*}
\int_{W^7} \kappa_2(x_1,x_2,x_3) \; \kappa_2(x_5,x_6,x_7) \; \varphi_1(x_3,x_5) \; \varphi_1(x_3,x_4) \; \varphi_1(x_4,x_5) \;\; \lambda^7 \, (\d {(x_1,\dots,x_7)}).
\end{align*}
We summarize some important properties of integral representations.

\begin{remark} \label{Bem:zmlr}
Let $m,l\in\{1,\dots,\alpha+1\}$ with $m\leq l$ and $l\leq r \leq m+l$.
Then the following statements hold.
\begin{multicols}{2}
\begin{enumerate}
\item[$\mathrm{(i)}$] $\z m l r (W) = \z l m r (W)$
\item[$\mathrm{(iii)}$] $\z m l l (W) = \z l l l (W)$
\item[$\mathrm{(ii)}$] $\z 1 1 1 (W) = |W|$, \quad $\z 1 1 2 (W) = |W|^2$
\item[$\mathrm{(iv)}$] $\z {m}{l}{m+l}(W) = \z m m m (W) \z l l l (W)$
\end{enumerate}
\end{multicols}
\end{remark}

While the first statement holds due to the isomorphism $K_{m,l}^r\cong K_{l,m}^r$, the second follows directly from the definition of $\varphi_0 \equiv 1$.
The third statement arises from the fact that one simplex is contained within the other in the representation (\ref{Def:Komplex_Kmlr}).
For the last statement, it is simply used that the two simplices from (\ref{Def:Komplex_Kmlr}) are disjoint in this case.

\begin{proposition} \label{Prop:Cov}
For $m,l\in\{1,\dots,\alpha+1\}$ with $m\leq l$ the following statements are true.
\begin{enumerate}
\item[$\mathrm{(i)}$] $\E{ \f {m-1} } \,=\, \frac{\beta^m}{m!} \z m m m (W)$
\item[$\mathrm{(ii)}$] $\Cov {\f {m-1}} {\f {l-1}} \,=\, \sum_{r=l}^{m+l-1} \frac{\beta^r}{(r-m)! (r-l)! (m+l-r)!} \; \z m l r (W)$
\end{enumerate}
\end{proposition}

\begin{proof}
    \begin{enumerate}
        \item[$\mathrm{(i)}$] We apply the multivariate Mecke equation (Theorem 4.4 in \cite{Last.Lectures}) to the Poisson process $\Psi_W$ and the function $f:(W\times\MM)^m\times\NDW\rightarrow [0,\infty)$, 
        \begin{align*}
            f\big((x_1,u_1),\dots,(x_m,u_m),\eta\big) \,:=\, \1\big\{ [x_{1},\dots,x_{m}]\in\Delta(\eta) \big\},
        \end{align*}
        where $\Delta(\eta)$ denotes the complex constructed from $\eta$.
        We obtain
        \begin{align*}
            \E{ \f {m-1} } \,&=\, \mathbb{E}\Bigg[ \frac{1}{m!} \sum_{((x_1,u_1),\dots,(x_m,u_m))\in\Psi_W^{(m)}} \1\big\{ [x_{1},\dots,x_{m}]\in\Delta_W \big\} \Bigg] \\
            &=\, \frac{\beta^m}{m!}\int_{W^m} \kappa_{m-1}(x_1,\dots,x_m) \; \lambda^m \, (\d {(x_1,\dots,x_m)}) \,=\, \frac{\beta^m}{m!} \z m m m (W).
        \end{align*}
        \item[$\mathrm{(ii)}$] We use the representation
        \begin{align*}
            \f {m-1} \f {l-1} \,=\, \sum_{r=l}^{m+l} \sum_{\sigma\in F_{m-1}(\Delta_W)} \sum_{\rho\in F_{l-1}(\Delta_W)} \1\big\{ \sigma,\rho\in\Delta_W, |\sigma\cup\rho| = r \big\}
        \end{align*}
        and write
        \begin{align*}
            T_{m,l}(A) \,:=\, \big\{ (\sigma,\rho) \mid \sigma,\rho\subseteq A, \sigma\cup\rho = A, |\sigma|=m, |\rho|=l \big\}
        \end{align*}
        for a set $A$ with $l\leq |A|\leq m+l$.
        Note that $|T_{m,l}(A)|=\binom{|A|}{m}\binom{m}{m+l-|A|}$.
        We obtain by another application of the multivariate Mecke equation
        \begin{align*}
            &\E{ \f {m-1} \f {l-1} } \\
            &\qquad=\, \sum_{r=l}^{m+l} \E{ \frac{1}{r!} \sum_{((x_1,u_1),\dots,(x_r,u_r))\in\Psi_W^{(r)}} \sum_{(\sigma,\rho)\in T_{m,l}(\{x_1,\dots,x_r\})}  \1\big\{ \sigma,\rho\in\Delta_W\big\} } \\
            &\qquad=\, \sum_{r=l}^{m+l} \frac{\beta^r}{r!} \binom{r}{m}\binom{m}{m+l-r} \int_{W^r} \kappa_{m,l}^r(x_1,\dots,x_r) \; \lambda^r \, (\d {(x_1,\dots,x_r)}) \\
            &\qquad=\, \sum_{r=l}^{m+l} \frac{\beta^r}{(r-m)!(r-l)!(m+l-r)!} \z m l r (W).
        \end{align*}
        By realizing that, for $r=m+l$, the summand equals
$\E{ \f {m-1}} \E{\f {l-1} }$ and using Remark~\ref{Bem:zmlr}(iv),
the statement for the covariances follows.
    \end{enumerate}
\end{proof}

\begin{corollary} \label{Kor:Moments_Euler}
The expectation and variance of the generalized Euler characteristic $\chi_a(\Delta_W)$ are given by
\begin{align*}
\E{ \chi_a(\Delta_W) } \,&=\, \sum_{m=1}^{\alpha+1} a_{m-1} \frac{\beta^m}{m!} \z m m m (W), \\
\V{ \chi_a(\Delta_W) } \,&=\, \sum_{m,l=1}^{\alpha+1} \sum_{r=\max\{m,l\}}^{m+l-1} a_{m-1} a_{l-1} \frac{\beta^r}{(r-m)!(r-l)!(m+l-r)!} \, \z m l r (W).
\end{align*}
\end{corollary}

\begin{proof}
    The assertions follow directly from Proposition \ref{Prop:Cov} writing the variance of a sum of random variables as a sum of covariances.
\end{proof}

\section{Quantitative normal approximation in the RCM} \label{Sec:NormalAppTheorem}

In this section, we lay the groundwork for proving central limit theorems in the upcoming sections.
First, in the initial subsection, we will prove a theorem that provides upper bounds for the distance of a standardized Poisson functional to a standard normally distributed random variable.
The second subsection then provides important tools to handle the quantities from these bounds.

\subsection{A Theorem for normal approximation}

The foundation for the central limit theorems that we will derive in the next sections is provided by Theorems 1.1 and 1.2 in \cite{Last.NormalApprox}.
These theorems provide upper bounds for the distance between a standardized Poisson functional and a standard normally distributed random variable with respect to the Wasserstein and Kolmogorov metrics.
However, in order to use these bounds appropriately, we need to modify the statements of the two theorems.
To do so, we aim to prove a corresponding statement in which the difference operators that appear in the quantities of Theorems 1.1 and 1.2 in \cite{Last.NormalApprox} are replaced by difference operators that are benign in a certain way.

To clarify this description, we denote by $T(\eta)$ the simplicial complex constructed from a counting measure $\eta\in\NDM$ as explained in Section \ref{Sec:Definition_RCM}.
In doing so, we refrain from formally defining a space of simplicial complexes, which is possible but not necessary here.
For a function $f$ defined for simplicial complexes with vertices in $\X$, the first-order difference operator $D_{(x,u)}f$ at a point $(x,u)\in\X\times\MM$ is defined by
\begin{align} \label{classical_1storder_DiffOp}
    D_{(x,u)}f \,:=\, f\big(T(\eta+\delta_{(x,u)})\big) - f(T(\eta)).
\end{align}
The difference operator $D_{(x,u)}f$ measures the influence of the added point $(x,u)$ on the functional.
Since the added point can influence the assignment of the random variables in the marks to the potential simplices (see Section \ref{Sec:Definition_RCM}), in general, $T(\eta)$ is not a subcomplex of $T(\eta+\delta_{(x,u)})$.
This significantly complicates working with these operators.
Theorems 1.1 and 1.2 from \cite{Last.NormalApprox} were already transferred to other operators in the RCM (see Theorem 6.1 in \cite{Last.RCM}), but here we will use an even more general construction.
\begin{figure}[t!]
  \captionsetup{ labelfont = {bf}, format = plain }
  \centering
  \subfloat[][]{\includegraphics[width=0.31\linewidth]{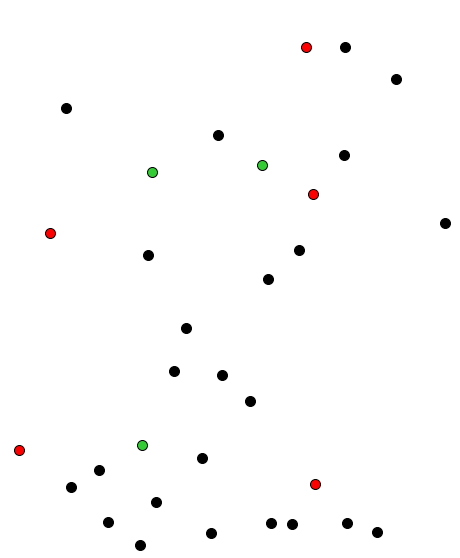}}
  \quad
  \subfloat[][]{\includegraphics[width=0.31\linewidth]{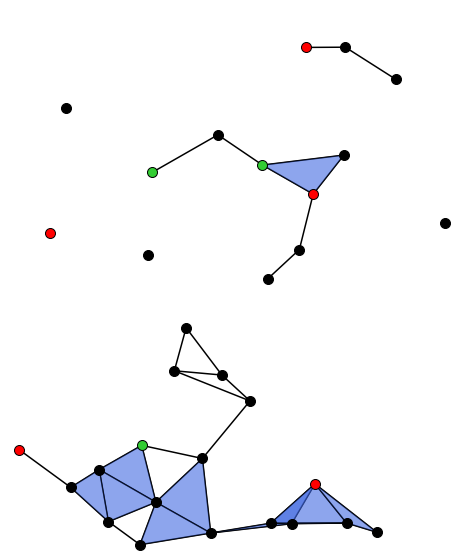}}
  \quad
  \subfloat[][]{\includegraphics[width=0.31\linewidth]{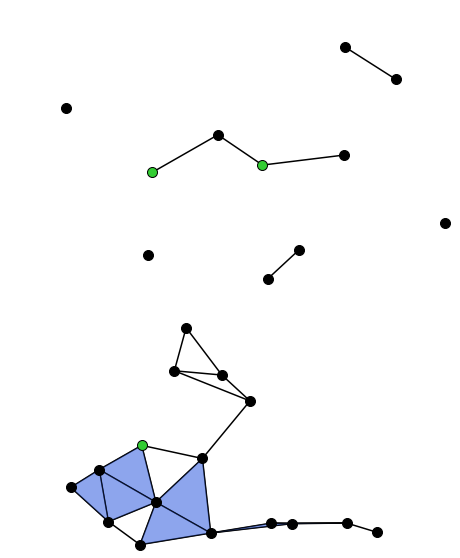}}
  \caption{The construction of the simplicial complex $\Delta^{x_1,\dots,x_l,I}$ in three steps for $l=8$ and $I=\{1,2,3\}$: (a) a realization of the Poisson process (black points) together with the additional points $x_1,x_2,x_3$ (green) and $x_4,\dots,x_8$ (red); (b) the simplicial complex $\Delta^{x_1,\dots,x_l}$ constructed from all points; and (c) the subcomplex $\Delta^{x_1,\dots,x_l,I}$ obtained after removing the red points.}
  \label{fig:Construction_DiffOp}
\end{figure}
Let $x_1,\dots,x_l\in\X$, $l\in\N$, and $U_1,\dots,U_l\sim\Q$ be i.i.d. random marks.
For a subset $I\subseteq\{1,\dots,l\}$, we denote by $\Delta^{x_1,\dots,x_l,I}$ the simplicial complex obtained by removing from $\Delta^{x_1,\dots,x_l}:=T(\Psi+\sum_{i=1}^l\delta_{(x_i,U_i)})$ all simplices that contain a vertex $x_j$ with $j\in\{1,\dots,l\}\setminus I$.
The additional points $x_j$, $j\in\{1,\dots,l\}\setminus I$, are thus initially used in the construction of the simplicial complex but are then removed from it again.
As a result, the simplicial complexes $\Delta^{x_1,\dots,x_l,I}$ are always subcomplexes of $\Delta^{x_1,\dots,x_l}$, which makes them, in a certain sense, coupled to each other.
Figure \ref{fig:Construction_DiffOp} illustrates the construction of the simplicial complexes $\Delta^{x_1,\dots,x_l,I}$ in three steps.
First, we add the points $x_1, \dots, x_l$ to the Poisson process.
Next, we build the simplicial complex $\Delta^{x_1,\dots,x_l}$ containing all points.
Finally, we remove the points $x_j$ for $j \in \{1, \dots, l\} \setminus I$.
While $\Delta^{x_1,\dots,x_l,I}\neq \Delta^{x_{i_1},\dots,x_{i_m}}$ for $I=\{i_1,\dots,i_m\}$ in general, they are identical in distribution.
For a function $f$ defined for simplicial complexes with vertices in $\X$, we define for $0\leq k\leq l$ the $k$-th order difference operator
\begin{align} \label{verallgDiffOp}
\Lambda_{(x_1,U_1),\dots,(x_l,U_l)}^k f(\Delta) \,:=\, \sum_{I\subseteq\{1,\dots,k\}} (-1)^{k-|I|} \, f\big( \Delta^{x_1,\dots,x_l,I} \big).
\end{align}
Note that in the case $k=0$, the operator is given by $f\big(\Delta^{x_1,\dots,x_l,\emptyset}\big)$, which has the same distribution as $f(\Delta)$.
Additional points that do not appear in any simplicial complex are introduced to facilitate the consideration of products of operators from (\ref{verallgDiffOp}).
For example, regarding a random variable of the form $\Lambda^1_{(x_1,U_1),(x_2,U_2)}f(\Delta)\Lambda^1_{(x_2,U_2),(x_1,U_1)}f(\Delta)$, all simplicial complexes appearing here, according to (\ref{verallgDiffOp}), are subcomplexes of $\Delta^{x_1,x_2}$.
As a result, calculations with these quantities are greatly simplified by this construction.
Additionally, it also facilitates the proof of Theorem \ref{Th:myRCM}.
To demonstrate the key property of the operators from (\ref{verallgDiffOp}), we also consider higher-order difference operators associated with the operators from (\ref{classical_1storder_DiffOp}), given by
\begin{align} \label{classical_DiffOp}
    D_{(x_1,U_1),\dots,(x_k,U_k)}^k f(\Delta) \,:=\, \sum_{I\subseteq\{1,\dots,k\}} (-1)^{k-|I|} \, f\big( T(\Psi+\sum_{i\in I}\delta_{(x_i,U_i)}) \big).
\end{align}
Although these operators are often introduced differently in the literature (see, e.g., (18.2) in \cite{Last.Lectures}), they can nonetheless be expressed as in (\ref{classical_DiffOp}) (cf. (18.3) in \cite{Last.Lectures}).
To make the notation a bit more flexible, we incorporate the decomposition $\QQ$ into the notation.
For $I=\{i_1,\dots,i_{s}\}$ and $\{j_{s+1},\dots,j_l\}=\{1,\dots,l\}\setminus I$ we write
\begin{align*}
    \Lambda_{(x_i,\HK U i)_{i\in I},(x_j,\HK U j)_{j\in [l]\setminus I}}^{|I|,\QQ} f(\Delta) \,&:=\, \Lambda_{(x_{i_1},U_{i_1}),\dots,(x_{i_s},U_{i_s}),(x_{j_{s+1}},U_{j_{s+1}}),\dots,(x_{j_l},U_{j_l})}^s f(\Delta), \\
    D_{(x_i,\HK U i)_{i\in I}}^{|I|,\QQ} f(\Delta) \,&:=\, D_{(x_{i_1},U_{i_1}),\dots,(x_{i_s},U_{i_s})}^s f(\Delta)
\end{align*}
to represent the operators from (\ref{verallgDiffOp}) and (\ref{classical_DiffOp}), respectively, where $\QQ$ is employed in constructing the simplicial complexes.
Now we fix two different decompositions $\QQ=(Q_i)_{i\in\N},\RR$ of $\X$ and denote by $m_i\in\N$ the unique index with $x_i\in Q_{m_i}$.
The crucial fact is that, whenever the sets $Q_{m_1},\dots,Q_{m_l}$ are pairwise disjoint, on the event 
\begin{align} \label{Ereignis_A_verallg}
A \,:=\, \big\{ \Phi\big( \cup_{i=1}^l Q_{m_i} \big)=0 \big\}
\end{align}
it is true that
\begin{align} \label{verallgDiffOp_verteilungsgleich}
\1_A \Big( D_{(x_i,\HK U i)_{i\in I}}^{|I|,\QQ} f(\Delta) \Big)_{\emptyset\neq I\subseteq [l]} \,\stackrel{d}{=}\, \1_A \Big( \Lambda_{(x_i,\HK U i)_{i\in I},(x_j,\HK U j)_{j\in [l]\setminus I}}^{|I|,\RR} f(\Delta) \Big)_{\emptyset\neq I\subseteq [l]}
\end{align}
as in this case all simplicial complexes appearing in the definition of the operators on the left-hand side are subcomplexes of $T(\Psi+\sum_{i=1}^l\delta_{(x_i,U_i)})$.
The idea now is to find a sequence of decompositions $(\QQ^m)_{m\in\N}$ of $\X$, as in the proof of Theorem 6.1 in \cite{Last.RCM}, such that $\P(A_m^c)\rightarrow 0$ as $m\rightarrow\infty$, where $A_m$ denotes the event $A$ from (\ref{Ereignis_A_verallg}) corresponding to the decomposition $\QQ^m$.
For this purpose we choose $\QQ^m=(Q_k^m )_{k\in\N}$ such that
\begin{align} \label{epsilon_Zerlegung}
    \lambda(Q_k^m) \,\leq\, \frac{1}{m}, \qquad Q^{m+1}(x)\subseteq Q^m(x), \qquad k,m\in\N, \; x\in\X
\end{align}
holds, where $Q^m(x)$ denotes the uniquely determined set of the decomposition $\QQ^m$ with $x\in Q^m(x)$.
In Section 4 of \cite{Last.RCM} it is shown that such a sequence of decompositions exists on Borel spaces.
All preparations for deriving a theorem for the difference operators $\Lambda^k$ based on the Theorems 1.1 and 1.2 in \cite{Last.NormalApprox} have been made.
The proof follows, in its basic structure, the proof of Theorem 6.1 in \cite{Last.RCM}, although the operators used differ.
We only consider functionals of the form $f(\Delta_W)$, as it will be functionals of this form for which we want to apply the theorem.
For two random variables $Z_1,Z_2$ we denote by
\begin{align*}
d_W(Z_1,Z_2) \,&:=\, \sup\limits_{h\in \mathrm{Lip(1)}} \big| \E{ h(Z_1)-h(Z_2) } \big|, \\
d_K(Z_1,Z_2) \,&:=\, \sup\limits_{t\in\R} \; | \P(Z_1\leq t) - \P(Z_2\leq t) |
\end{align*}
the Wasserstein distance and the Kolmogorov distance, respectively, where $\mathrm{Lip}(1)$ denotes the set of all Lipschitz continuous functions $h:\R\rightarrow\R$ with Lipschitz constant at most one.
We will consider functionals of the form $f(\Delta_W)$, where $f$ is defined for finite simplicial complexes with vertices in $\X$, and $W$ is an observation window.
This can be understood as the composition $\Delta \mapsto \Delta_W \mapsto f(\Delta_W)$.

\begin{theorem} \label{Th:myRCM}
Let $W\in\Bf$ and $F=f(\Delta_W)$ with $\E{F}=0$, $\V{F}=1$ and $\E{F^4}<\infty$, where $f$ is a function defined for finite simplicial complexes with vertices in $\X$.
Moreover, assume that $\gamma_5,\gamma_6<\infty$.
Then it is true that
\begin{align*}
d_W(F,N) &\,\leq\, \gamma_1 + \gamma_2 + \gamma_3, \\
d_K(F,N) &\,\leq\, \gamma_1 + \gamma_2 + \gamma_3 + \gamma_4 + \gamma_5 + \gamma_6,
\end{align*}
where the quantities $\gamma_1,\dots,\gamma_6$ are defined by
\begin{align*}
\gamma_1 &\,:=\, 2 \Bigg[ \beta^3 \int_{W^3} \E{ \left( \Lambda_{(x_1,U_1),(x_2,U_2)}^1F \right)^2 \left( \Lambda_{(x_2,U_2),(x_1,U_1)}^1F \right)^2}^\frac{1}{2} \\
& \qquad \enskip \times \, \E{ \left(\Lambda_{(x_1,U_1),(x_3,U_3),(x_2,U_2)}^2 F\right)^2 \left(\Lambda_{(x_2,U_2),(x_3,U_3),(x_1,U_1)}^2 F\right)^2}^\frac{1}{2}  \; \lambda^3 \, (\d {(x_1,x_2,x_3)}) \Bigg]^\frac{1}{2} \hspace{-6pt}, \\
\gamma_2 &\,:= \left[ \beta^3 \hspace{-2pt} \int_{W^3} \hspace{-2pt} \E{ \left(\Lambda_{(x_1,U_1),(x_3,U_3),(x_2,U_2)}^2 F\right)^2 \left(\Lambda_{(x_2,U_2),(x_3,U_3),(x_1,U_1)}^2 F\right)^2} \lambda^3 \, (\d {(x_1,x_2,x_3)}) \right]^\frac{1}{2}  \hspace{-6pt}, \\
\gamma_3 &\,:=\, \beta \int_W \E{ |\Lambda_{(x,U)}F|^3 } \; \lambda \, (\d x), \\
\gamma_4 &\,:=\, \frac{\beta}{2} \E{ F^4 }^\frac{1}{4} \int_W \E{ \left(\Lambda_{(x,U)}F\right)^4 }^\frac{3}{4} \; \lambda \, (\d x), \\
\gamma_5 &\,:=\, \left[ \beta \int_W \E{ \left(\Lambda_{(x,U)}F\right)^4 } \; \lambda \, (\d x) \right]^\frac{1}{2} \hspace{-6pt}, \\
\gamma_6 &\,:=\, \Bigg[ \beta^2 \int_{W^2} 6\; \E{ \left(\Lambda_{(x_1,U_1)}F\right)^4 }^\frac{1}{2} \E{ \left(\Lambda_{(x_1,U_1),(x_2,U_2)}^2 F\right)^4 }^\frac{1}{2} \\
& \qquad \qquad \enskip + 3\; \E{ \left(\Lambda_{(x_1,U_1),(x_2,U_2)}^2 F\right)^4 } \; \lambda^2 \, (\d {(x_1,x_2)}) \Bigg]^\frac{1}{2}
\end{align*}
with independent marks $U,U_1,U_2,U_3\sim\Q$, which are also independent of $\Psi$.
Furthermore, we have $\gamma_1,\dots,\gamma_4<\infty$.
\end{theorem}

\begin{proof}
For $m\in\N$ we fix a decomposition $\QQ^m$ with the properties from (\ref{epsilon_Zerlegung}) and denote by $F_m$ the functional $f(\Delta_W)$, where the simplicial complex is constructed with respect to the decomposition $\QQ^m$. 
Initially, we assume $f$ is bounded.
We apply the Theorems 1.1 and 1.2 in \cite{Last.NormalApprox} to the Poisson process $\Psi_W$ and the Poisson functional $F_m$ with $F_m\stackrel{d}{=}F$.
Note that the theorems are applicable since the integrability condition is satisfied due to the boundedness of $f$ and $|W|<\infty$. 
Using H\"older's inequality, we move the integration with respect to the mark distribution into the expected values and obtain
\begin{align*}
d_W(F,N) \,&\leq\, \gamma_1^*(m) + \gamma_2^*(m) + \gamma_3^*(m), \\
d_K(F,N) \,&\leq\, \gamma_1^*(m) + \gamma_2^*(m) + \gamma_3^*(m) + \gamma_4^*(m) + \gamma_5^*(m) + \gamma_6^*(m),
\end{align*}
with the quantities
\begin{align*}
\gamma_1^*(m) &\,:=\, 2 \Bigg[ \beta^3 \int_{W^3} \E{\left( D_{(x_1,U_1)}F_m \right)^2 \left( D_{(x_2,U_2)} F_m \right)^2 }^\frac{1}{2} \\ 
& \hspace*{30pt} \times\, \E{ \left( D_{(x_1,U_1),(x_3,U_3)}^2 F_m \right)^2 \left( D_{(x_2,U_2),(x_3,U_3)}^2 F_m\right)^2 }^\frac{1}{2}  \; \lambda^3 \, (\d {(x_1,x_2,x_3)})  \Bigg]^\frac{1}{2}, \\
\gamma_2^*(m) &\,:=\, \left[ \beta^3 \int_{W^3} \E{ \left( D_{(x_1,U_1),(x_3,U_3)}^2 F_m\right)^2 \left( D_{(x_2,U_2),(x_3,U_3)}^2 F_m\right)^2 }\; \lambda^3 \, (\d {(x_1,x_2,x_3)}) \right]^\frac{1}{2}, \\
\gamma_3^*(m) &\,:=\, \beta \int_W \E{ |D_{(x,U)}F_m|^3 } \; \lambda \, (\d x), \\
\gamma_4^*(m) &\,:=\, \frac{\beta}{2} \E{ F_m^4 }^\frac{1}{4} \int_W \E{ \left( D_{(x,U)}F_m\right)^4 }^\frac{3}{4} \; \lambda \, (\d x), \\
\gamma_5^*(m) &\,:=\, \left[ \beta \int_W \E{ \left( D_{(x,U)}F_m\right)^4 } \; \lambda \, (\d x) \right]^\frac{1}{2}, \\
\gamma_6^*(m) &\,:=\, \Bigg[ \beta^2 \int_{W^2} 6\; \E{ \left( D_{(x_1,U_1)} F_m \right)^4 }^\frac{1}{2} \E{ \left( D_{(x_1,U_1),(x_2,U_2)}^2 F_m\right)^4 }^\frac{1}{2} \\ 
& \qquad \qquad \qquad + 3\; \E{ \left( D_{(x_1,U_1),(x_2,U_2)}^2 F_m\right)^4 } \; \lambda^2 \, (\d {(x_1,x_2)}) \Bigg]^\frac{1}{2}.
\end{align*}
If $A_m$ denotes the event from (\ref{Ereignis_A_verallg}) when using the decomposition $\QQ^m$, then by (\ref{epsilon_Zerlegung}), the convergence $\P(A_m^c)\rightarrow 0$ follows.
Due to property (\ref{verallgDiffOp_verteilungsgleich}) and the boundedness of $f$, the integrand in $\gamma_i^*(m)$ converges pointwise almost everywhere to the integrand of $\gamma_i$ for $i\in\{1,\dots,6\}$.
Since the integrands are themselves bounded due to the boundedness of $f$ and $|W|<\infty$, the dominated convergence theorem shows $\gamma_i^*(m)\rightarrow\gamma_i$ for $m\rightarrow\infty$.

In the following, we do not assume that $f$ and thus $F$ are bounded anymore. 
For $n\in \N$, we define a bounded functional $F_n:=f_n(\Delta_W)$ by
\begin{align} \label{Definition_approxFct}
f_n(\Delta_W) \,:=\, \1\big\{ |f(\Delta_W)|\leq n \big\}\, f(\Delta_W) + \1\big\{ f(\Delta_W)> n \big\}\, n - \1\big\{ f(\Delta_W)< -n \big\}\, n.
\end{align}
Then $F_n\rightarrow F$ holds as $n\rightarrow\infty$ both $\P$-almost surely and in $L^4(\P)$. 
We define $\gamma_i^{(n)}$ as $\gamma_i$, where the functional $F$ is replaced by $F_n-\E{F_n}$. 
The term $\E{F_n}$ has no impact on the difference operators $\Lambda^k$, as these are linear and vanish on constants. 
To show the convergence of $\gamma_i^{(n)}$ towards $\gamma_i$ as $n\rightarrow\infty$ for $i\in\{1,\dots,6\}$, we use the estimates
\begin{align} 
|\Lambda_{(x_1,U_1),\dots,(x_l,U_l)}^1F_n| \,&\leq\, |\Lambda_{(x_1,U_1),\dots,(x_l,U_l)}^1F| + 2|\Lambda_{(x_1,U_1),\dots,(x_l,U_l)}^0F|, \label{NA3} \\
|\Lambda_{(x_1,U_1),\dots,(x_l,U_l)}^2F_n| \,&\leq\, |\Lambda_{(x_1,U_1),\dots,(x_l,U_l)}^2F| + 2|\Lambda_{(x_1,U_1),\dots,(x_l,U_l)}^1F| \nonumber \\
& \qquad + 2|\Lambda_{(x_2,U_2),(x_1,U_1),(x_3,U_3),\dots,(x_l,U_l)}^1F| + 4|\Lambda_{(x_1,U_1),\dots,(x_l,U_l)}^0F|, \label{NA4}
\end{align}
which can be shown in the same manner as the analogous estimates (6.6) and (6.7) in \cite{Last.RCM}.
The integrability condition $\gamma_5,\gamma_6<\infty$ also implies the finiteness of $\gamma_1,\gamma_2,\gamma_3,\gamma_4$.
Hence, in (\ref{NA3}) and (\ref{NA4}), we obtain integrable upper bounds for the integrands from the quantities $\gamma_1^{(n)},\dots,\gamma_6^{(n)}$ for almost all added points.
Note that
\begin{align*}
    \Lambda_{(x_1,U_1),\dots,(x_l,U_l)}^k g(\Delta) \,\stackrel{d}{=}\, \Lambda_{(x_1,U_1),\dots,(x_k,U_k)}^k g(\Delta) 
\end{align*}
holds for any function $g$ and $k\leq l$.
Therefore, the estimates (\ref{NA3}), (\ref{NA4}) allow for an application of the dominated convergence theorem to see that the integrands converge almost everywhere to the corresponding integrands in the quantities $\gamma_1,\dots,\gamma_6$.
A further application of the dominated convergence theorem, together with $\E{ (F_n-\mathbb{E}[F_n])^4 }\rightarrow\E{F^4}$, yields
\begin{align} \label{Konvergenz_gamma_n}
\lim\limits_{n\rightarrow\infty} \gamma_i^{(n)} \,=\, \gamma_i, \qquad i\in\{1,\dots,6\}.
\end{align}
We define
\begin{align*}
\widetilde{F}_n \,:=\, \frac{F_n-\E{F_n}}{\sqrt{\V{F_n}}}
\end{align*}
for all $n\in\N$ with $\V{F_n}>0$.
For these functionals, the statement of the theorem holds according to the first part of the proof for bounded functionals.
The desired statement now follows from Lemma 6.2 in \cite{Last.RCM} as well as (\ref{Konvergenz_gamma_n}) and $\mathbb{V}(F_n)\rightarrow\V{F}=1$.
\end{proof}

\subsection{Further tools}

To estimate the quantities from Theorem \ref{Th:myRCM}, we need, among other things, to estimate the fourth moment of the functional under consideration.
This can actually be done manually for the Euler characteristic, as shown in detail in the proof of Proposition 4.5 of \cite{Pabst.Thesis}.
This proof, primarily combinatorial in nature, demonstrates that the fourth centered moment of the Euler characteristic is a linear combination of integral representations of simplicial complexes with at most two connected components and no more than $4\alpha+2$ vertices.
However, if we are only interested in an appropriate estimate for proving a central limit theorem, the labor-intensive proof can be bypassed using Lemma 4.2 in \cite{Last.NormalApprox}.

\begin{corollary} \label{Kor:4thMoment}
    In the situation of Theorem \ref{Th:myRCM} we have
    \begin{align*}
        \E{F^4} \,\leq\, \max\bigg\{ 256\left[ \beta\int_W \E{ \left(\Lambda_{(x,U)}F\right)^4 }^\frac{1}2 \lambda \, (\d x)  \right]^2 \hspace*{-3pt},\, 4 \beta \int_W \E{ \left(\Lambda_{(x,U)}F\right)^4 } \, \lambda \, (\d x) + 2 \bigg\}.
    \end{align*}
\end{corollary}

\begin{proof}
    For bounded functions $f$, the statement follows from Lemma 4.2 in \cite{Last.NormalApprox} in the same way that Theorem \ref{Th:myRCM} follows from Theorems 1.1 and 1.2 in \cite{Last.NormalApprox}.
    In the general case, we first apply the statement for the functionals $F_n = f_n(\Delta_W)$, $n \in\N$, as defined in (\ref{Definition_approxFct}).
    Letting $\gamma_n$ denote the upper bound obtained with respect to $F_n$, we can demonstrate that $\gamma_n \rightarrow \gamma$ as $n \rightarrow \infty$, exactly as (\ref{Konvergenz_gamma_n}) was shown in the proof of Theorem \ref{Th:myRCM}, where $\gamma$ is the right-hand side of the claimed statement.
    This now follows directly from $\E{F_n^4} \rightarrow \E{F^4}$ as $n \rightarrow \infty$.
\end{proof}

To estimate the quantities $\gamma_1,\dots,\gamma_6$ from Theorem \ref{Th:myRCM} for the (standardized) Euler characteristic, we need to establish lower bounds for the variance of the Euler characteristic.
For this purpose, we want to prove a Fock space representation for the variance of Poisson functionals concerning the operators from (\ref{verallgDiffOp}) instead of the ones from (\ref{classical_DiffOp}) (Theorem 18.6 in \cite{Last.Lectures}).
In fact, it is even possible to show an inequality of this form, which provides a lower bound for the variance, for general functionals as in Theorem \ref{Th:myRCM} (see Theorem 3.15 in \cite{Pabst.Thesis}).
However, this proof is very technical, which is why we prefer to simply recalculate the Fock space representation for the Euler characteristic here.
Before that, we would like to take a look at the difference operators from (\ref{verallgDiffOp}) for the Euler characteristic.
To do so, we denote by $f_i^{x_1,\dots,x_k}(K)$ the number of $i$-simplices in $K$ that include the points $x_1,\dots,x_k$.

\begin{lemma} \label{Lem:EulerDiffOp}
For $k\leq l$ and pairwise different $x_1,\dots,x_l\in\X$, it holds $\P$-almost surely that
\begin{align*}
\Lambda_{(x_1,U_1),\dots,(x_l,U_l)}^k \chi_a(\Delta_W) \,=\, \sum_{i=k-1}^\alpha a_i \, f_i^{x_1,\dots,x_k} \big(\Delta_W^{x_1,\dots,x_l,\{1,\dots,k\}}\big).
\end{align*}
\end{lemma}

\begin{proof}
All transformations within this proof are to be understood $\P$-almost surely.
First, by definition of the difference operators and the Euler characteristic, we have
\begin{align} \label{EulerDiff}
\Lambda_{(x_1,U_1),\dots,(x_l,U_l)}^k \chi_a(\Delta_W) \,&=\, \sum_{I\subseteq \{1,\dots,k\}} (-1)^{k-|I|} \; \chi_a\big(\Delta_W^{x_1,\dots,x_l,I}\big) \nonumber \\
&=\, \sum_{i=0}^\alpha a_i \sum_{I\subseteq \{1,\dots,k\}} (-1)^{k-|I|} \; f_i\big(\Delta_W^{x_1,\dots,x_l,I}\big).
\end{align}
In the next step, we fix a simplex $\sigma\in\Delta_W^{x_1,\dots,x_l}$ with $x_j\notin\sigma$ for all $j\in\{k+1,\dots,l\}$ and consider how often this simplex is counted by the above sum.
To accomplish this, let $J\subseteq\{1,\dots,k\}$ denote the unique set such that $x_j\in\sigma$ if and only if $j\in J$ for $j\in\{1,\dots,k\}$.
The simplex $\sigma$ is counted in the summand for $i=\dim(\sigma)$ and $I\subseteq\{1,\dots,k\}$ in (\ref{EulerDiff}) if and only if $J\subseteq I$ holds.
Thus, the total contribution of $\sigma$ to the second sum in (\ref{EulerDiff}) for $i=\dim(\sigma)$ is
\begin{align*}
\sum_{I\subseteq \{1,\dots,k\}} (-1)^{k-|I|} \; \1\{ J\subseteq I \} &= \sum_{J\subseteq I\subseteq \{1,\dots,k\}} (-1)^{k-|I|} = \sum_{I\subseteq\{1,\dots,k-|J|\}} (-1)^{k-|J|-|I|} \\
&= \sum_{q=0}^{k-|J|} \binom {k-|J|} {q} (-1)^{k-|J|-q} =
\begin{cases}
1, & J=\{1,\dots,k\}, \\
0, & \text{otherwise}. \\
\end{cases}
\end{align*}
Hence, the operator $\Lambda_{(x_1,U_1),\dots,(x_l,U_l)}^k \chi_a(\Delta_W)$ counts only those simplices of $\Delta_W^{x_1,\dots,x_l}$ (with the corresponding prefactor) that include all added points $x_1,\dots,x_k$, but none of the points $x_{k+1},\dots,x_l$.
The assertion follows since such a simplex must have at least dimension $k-1$.
\end{proof}

In the next step, we derive a Fock space representation (cf. Theorem 18.6 in \cite{Last.Lectures}) for the variance of the Euler characteristic with respect to the operators from (\ref{verallgDiffOp}).

\begin{proposition} \label{Prop:EulerFock}
The variance $\V{\chi_a(\Delta_W)}$ is equal to
\begin{align*}
\sum_{k=1}^\infty \frac{\beta^k}{k!} \hspace{-2pt} \int_{W^k} \hspace{-4pt} \E{ \mathbb{E}\Big[ \Lambda_{(x_1,U_1),\dots,(x_k,U_k)}^{k} \chi_a(\Delta_W) \,\Big|\, \1\big\{ [x_1,\dots,x_k]\in\Delta_W^{x_1,\dots,x_k} \big\} \Big]^2 } \hspace{-1pt} \lambda^k (\d {(x_1,\dots,x_k})).
\end{align*}
\end{proposition}

\begin{proof}
First, note that, as a consequence of Lemma \ref{Lem:EulerDiffOp}, all summands of this infinite sum for $k>\alpha+1$ are zero.
From Corollary \ref{Kor:Moments_Euler}, we obtain with the substitution $k=m+l-r$
\begin{align}
\V{ \chi_a(\Delta_W) } \,&=\, \sum_{m,l=1}^{\alpha+1} \sum_{k=1}^{\min\{m,l\}} a_{m-1} a_{l-1} \frac{\beta^{m+l-k}}{(m-k)!(l-k)!k!} \, \z m l {m+l-k} (W) \nonumber \\
&=\, \sum_{k=1}^{\alpha+1} \sum_{m,l=k}^{\alpha+1} a_{m-1} a_{l-1} \frac{\beta^{m+l-k}}{(m-k)!(l-k)!k!} \, \z m l {m+l-k} (W). \label{EulerVar_rewritten}
\end{align}
We will show that for $k\leq\alpha+1$ the $k$-th summand of (\ref{EulerVar_rewritten}) coincides with the $k$-th summand of the infinite sum of the claim.
Let $k\leq\alpha+1$ be fixed.
Lemma \ref{Lem:EulerDiffOp} demonstrates that the conditional expectation appearing in the $k$-th summand of the claim is $\P$-almost surely equal to
\begin{align*}
\1\big\{ [x_1,\dots,x_k]\in\Delta_W^{x_1,\dots,x_k} \big\} \sum_{m=k}^{\alpha+1} a_{m-1} \E{ f_{m-1}^{x_1,\dots,x_k} \big(\Delta_W^{x_1,\dots,x_k}\big) \,\big|\, [x_1,\dots,x_k]\in\Delta_W^{x_1,\dots,x_k} }.
\end{align*}
For $\kappa_{k-1}(x_1,\dots,x_k)=0$ this random variable is $\P$-almost surely zero.
Otherwise, an application of Mecke's formula shows that
\begin{align}
&\E{ f_{m-1}^{x_1,\dots,x_k} \big(\Delta_W^{x_1,\dots,x_k}\big) \,\big|\, [x_1,\dots,x_k]\in\Delta_W^{x_1,\dots,x_k} } \nonumber \\
&\qquad\qquad =\, \frac{\beta^{m-k}}{(m-k)!} \int_{W^{m-k}} \frac{\kappa_{m-1}(x_1,\dots,x_k,y_{k+1},\dots,y_{m})}{\kappa_{k-1}(x_1,\dots,x_k)} \; \lambda^{m-k} \, (\d {(y_{k+1},\dots,y_{m})}). \label{Integral_of_help}
\end{align}
Denoting the integral on the right-hand side of (\ref{Integral_of_help}) by $g_m(x_1,\dots,x_k)$, we obtain
\begin{align*}
&\E{ \mathbb{E}\Big[ \Lambda_{(x_1,U_1),\dots,(x_k,U_k)}^{k} \chi_a(\Delta_W) \;\big|\; \1\big\{ [x_1,\dots,x_k]\in\Delta_W^{x_1,\dots,x_k} \big\} \Big]^2 } \\
&\qquad =\, \E{ \1\big\{ [x_1,\dots,x_k]\in\Delta_W^{x_1,\dots,x_k} \big\}
\left( \sum_{m=k}^{\alpha+1} a_{m-1}\frac{\beta^{m-k}}{(m-k)!}g_m(x_1,\dots,x_k) \right)^2 } \\
&\qquad =\, \sum_{m,l=k}^{\alpha+1} a_{m-1}a_{l-1} \frac{\beta^{m+l-2k}}{(m-k)!(l-k)!}
g_m(x_1,\dots,x_k)g_l(x_1,\dots,x_k)\kappa_{k-1}(x_1,\dots,x_k).
\end{align*}
Recognizing that according to Definition \ref{Def:Simplexfkt}
\begin{align*}
& g_m(x_1,\dots,x_k)g_l(x_1,\dots,x_k)\kappa_{k-1}(x_1,\dots,x_k) \\
&\qquad =\, \int_{W^{m+l-2k}} \kappa_{m,l}^{m+l-k}(x_1,\dots,x_k,y_{k+1},\dots,y_{m+l-k}) \; \lambda^{m+l-2k} \, (\d {(y_{k+1},\dots,y_{m+l-k})})
\end{align*}
we derive
\begin{align*}
&\frac{\beta^k}{k!} \int_{W^k} \E{ \mathbb{E}\Big[ \Lambda_{(x_1,U_1),\dots,(x_k,U_k)}^{k} \chi_a(\Delta_W) \;\Big|\; \1\big\{ [x_1,\dots,x_k]\in\Delta_W^{x_1,\dots,x_k} \big\} \Big]^2 } \; \lambda^k \, (\d {(x_1,\dots,x_k})) \\
&=\, \sum_{m,l=k}^{\alpha+1} \frac{a_{m-1}a_{l-1}\beta^{m+l-k}}{(m-k)!(l-k)!k!} \int_{W^{m+l-k}} \kappa_{m,l}^{m+l-k}(x_1,\dots,x_{m+l-k}) \, \lambda^{m+l-k} (\d {(x_1,\dots,x_{m+l-k}})).
\end{align*}
Since the last expression coincides with the $k$-th summand of (\ref{EulerVar_rewritten}), the desired statement follows.
\end{proof}

Since all summands in the representation of $\V{\chi_a(\Delta_W)}$ from Proposition \ref{Prop:EulerFock} are non-negative and according to the proof of Proposition \ref{Prop:EulerFock} the $k$-th summand coincides with the $k$-th summand from (\ref{EulerVar_rewritten}), for each $k\in\{1,\dots,\alpha+1\}$ we get
\begin{align} \label{Vorschlaege_lowerBound}
\V{\chi_a(\Delta_W)} \,\geq\, \sum_{m,l=k}^{\alpha+1} a_{m-1} a_{l-1} \frac{\beta^{m+l-k}}{(m-k)!(l-k)!k!} \, \z m l {m+l-k} (W).
\end{align}
Unfortunately, the individual summands on the right-hand side of (\ref{Vorschlaege_lowerBound}) could be negative.
Moreover, we would prefer a lower bound with as few summands as possible.
These two considerations suggest choosing $k=\alpha+1$, which leads to the estimate
\begin{align} \label{lowerBoundEulerVar}
\V{\chi_a(\Delta_W)} \,\geq\,  a_{\alpha}^2 \frac{\beta^{\alpha+1}}{(\alpha+1)!} \, \z {\alpha+1} {\alpha+1} {\alpha+1} (W).
\end{align}
Interestingly, this lower bound is identical to the expected value $\E{ \f {\alpha} }$, except for the factor $a_\alpha^2$, according to Proposition \ref{Prop:Cov}.

\section{Central limit theorems for the Euler characteristic} \label{Sec:CLT}

In this section, we establish a first central limit theorem for the Euler characteristic.
For this purpose, we fix a coefficient vector $\0 \neq a = (a_0, \dots, a_\alpha) \in \R^{\alpha+1}$ with $a_\alpha \neq 0$ for the entire section.
The latter is not a loss of generality, since otherwise $\alpha$ can simply be chosen to be smaller.
The proofs of the central limit theorems are based on Theorem \ref{Th:myRCM}, which we apply to the standardized Euler characteristic
\begin{align*}
\tilde{\chi}_a(\Delta_W) \,:=\, \frac{\chi_a(\Delta_W)-\E{\chi_a(\Delta_W)}}{\sqrt{\V{\chi_a(\Delta_W)}}},
\end{align*}
where $W \in \Bf$ is an observation window.
Since the difference operators from (\ref{verallgDiffOp}) are linear and vanish on constants, it follows that
\begin{align*}
\Lambda_{(x_1,U_1),\dots,(x_l,U_l)}^k\tilde{\chi}_a(\Delta_W) \,=\, \frac{\Lambda_{(x_1,U_1),\dots,(x_l,U_l)}^k \chi_a(\Delta_W)}{\sqrt{\V{\chi_a(\Delta_W)}}}.
\end{align*}
The key step for proving the central limit theorems will be to estimate the quantities from Theorem \ref{Th:myRCM}.

\subsection{Increasing intensity} \label{Sec:Intensity}

We fix an observation window $W\in\Bf$ and consider the limiting scenario of increasing intensity $\beta \rightarrow \infty$ for the Euler characteristic $\chi_a(\Delta_W)$.
We begin by estimating the expectations that appear in the expressions of Theorem \ref{Th:myRCM}.

\begin{lemma} \label{Lem:EulerIntensity}
For $F=\chi_a(\Delta_W)$ and pairwise different $x_1,x_2,x_3\in W$ it is true that
\begin{align*}
\E{ \left(\Lambda_{(x_1,U_1),(x_3,U_3),(x_2,U_2)}^2 F\right)^2 \left(\Lambda_{(x_2,U_2),(x_3,U_3),(x_1,U_1)}^2 F\right)^2} \,&=\, \mathcal{O}\left(\beta^{4\alpha-4}\right), \\[10pt]
\E{ \left(\Lambda_{(x_1,U_1),(x_2,U_2)}^2 F\right)^4 } \,&=\, \mathcal{O}\left(\beta^{4\alpha-4}\right), \\[10pt]
\E{ \left( \Lambda_{(x_1,U_1),(x_2,U_2)}^1F \right)^2 \left( \Lambda_{(x_2,U_2),(x_1,U_1)}^1F \right)^2} &= \mathcal{O}\left(\beta^{4\alpha}\right), \\[10pt]
\E{ \left(\Lambda_{(x_1,U)}F\right)^4 } \,=\, \mathcal{O}\left(\beta^{4\alpha}\right), \qquad\qquad
\E{ |\Lambda_{(x_1,U)}F|^3 } &= \mathcal{O}\left(\beta^{3\alpha}\right).
\end{align*}
\end{lemma}

\begin{proof}
We will demonstrate the detailed calculation for the first equation, as the expected value appearing there is also the most complex.
The remaining estimates follow analogously.
For this, we abbreviate by writing $\tau := \Phi(W) \sim \mathrm{Po}(\beta|W|)$ for the number of vertices of $\Delta_W$.
Then, using the representation of the difference operators from Lemma \ref{Lem:EulerDiffOp}, we have
\begin{align*}
&\E{ \left(\Lambda_{(x_1,U_1),(x_3,U_3),(x_2,U_2)}^2 F\right)^2 \left(\Lambda_{(x_2,U_2),(x_3,U_3),(x_1,U_1)}^2 F\right)^2} \\
&\qquad =\, \E{\bigg| \sum_{i,j,m,l=1}^\alpha \,a_ia_ja_ma_l \; \prod_{k_1\in\{i,j\}} f_{k_1}^{x_1,x_3}(\Delta_W^{x_1,x_2,x_3}) \;\prod_{k_2\in\{m,l\}} f_{k_2}^{x_2,x_3}(\Delta_W^{x_1,x_2,x_3}) \bigg|} \\
&\qquad\leq\, \sum_{i,j,m,l=1}^\alpha \, |a_ia_ja_ma_l| \; \E{ \prod_{k_1\in\{i,j\}} f_{k_1}^{x_1,x_3}(\Delta_W^{x_1,x_2,x_3}) \prod_{k_2\in\{m,l\}} f_{k_2}^{x_2,x_3}(\Delta_W^{x_1,x_2,x_3}) } \\
&\qquad\leq\, \sum_{i,j,m,l=1}^\alpha \, |a_ia_ja_ma_l| \; \E{ \binom \tau {i-1} \binom \tau {j-1} \binom \tau {m-1} \binom \tau {l-1} } \\
&\qquad\leq\, \sum_{i,j,m,l=1}^\alpha \, |a_ia_ja_ma_l| \; \E{ \tau^{i+j+m+l-4} }.
\end{align*}
The desired result follows from the fact that the $k$-th moment of a Poisson distribution is a polynomial of degree $k$ in its parameter, namely the $k$-th complete Bell polynomial.
\end{proof}

Note that it follows from the proof that the estimates in Lemma \ref{Lem:EulerIntensity} hold uniformly in $x_1,x_2,x_3\in W$.
Using Lemma \ref{Lem:EulerIntensity}, we can not only estimate the integrals from the quantities in Theorem \ref{Th:myRCM}, but, thanks to Corollary \ref{Kor:4thMoment}, also the fourth centered moment of the Euler characteristic.
With this, all preparations for proving an initial central limit theorem for the Euler characteristic are complete.

\begin{theorem} [CLT for increasing intensity] \label{Th:ZGWS1}
Let $a=(a_0,\dots,a_\alpha) \in \R^{\alpha+1}$ with $a_\alpha\neq 0$ and assume $\z {\alpha+1} {\alpha+1} {\alpha+1}(W)>0$.
Then
\begin{align*}
\frac{\chi_a(\Delta_W)-\E{ \chi_a(\Delta_W) }}{\sqrt{\V{ \chi_a(\Delta_W) }}} \;\xrightarrow{d}\; \mathcal{N}(0,1) \qquad \text{for } \beta \rightarrow \infty.
\end{align*}
Specifically, the quantities from Theorem \ref{Th:myRCM} applied to $\tilde{\chi}_a(\Delta_W)$ satisfy
\begin{align*}
\gamma_1,\gamma_3,\gamma_4,\gamma_5 = \mathcal{O}\left(\beta^{-\frac{1}{2}}\right), \qquad \gamma_2 = \mathcal{O}\left(\beta^{-\frac{3}{2}}\right), \qquad \gamma_6 = \mathcal{O}\left(\beta^{-1}\right).
\end{align*}
\end{theorem}

\begin{proof}
We apply Theorem \ref{Th:myRCM} to the standardized Euler characteristic $\tilde{\chi}_a(\Delta_W)$.
Lemma \ref{Lem:EulerIntensity} ensures that the necessary integrability conditions are satisfied.
First, from Corollary \ref{Kor:4thMoment} and Lemma \ref{Lem:EulerIntensity}, it can be deduced that
\begin{align*}
    \E{ \tilde{\chi}_a(\Delta_W)^4 } \,=\, \frac{\mathcal{O}(\beta^{4\alpha+2})}{\V{\chi_a(\Delta_W)}^2}.
\end{align*}
Since, according to Corollary \ref{Kor:Moments_Euler}, $\V{\chi_a(\Delta_W)}$ is a polynomial of degree $2\alpha+1$ with respect to $\beta$ due to $\z {\alpha+1} {\alpha+1} {\alpha+1}(W)>0$, the fourth moment of the standardized Euler characteristic is bounded with respect to $\beta$.
With the estimates from Lemma \ref{Lem:EulerIntensity}, we further obtain
\begin{alignat*}{6}
    &\gamma_1 \,&&=\, \frac{\mathcal{O}(\beta^{2\alpha+\frac{1}{2}})}{\V{\chi_a(\Delta_W)}} \,&&=\, \mathcal{O}(\beta^{-\frac{1}{2}}), \qquad &&\gamma_2 \,&&=\, \frac{\mathcal{O}(\beta^{2\alpha-\frac{1}{2}})}{\V{\chi_a(\Delta_W)}} \,&&=\, \mathcal{O}(\beta^{-\frac{3}{2}}), \\
    &\gamma_3 \,&&=\, \frac{\mathcal{O}(\beta^{3\alpha+1})}{\V{\chi_a(\Delta_W)}^\frac{3}{2}} \,&&=\, \mathcal{O}(\beta^{-\frac{1}{2}}), \qquad &&\gamma_4 \,&&=\, \frac{\mathcal{O}(\beta^{3\alpha+1})}{\V{\chi_a(\Delta_W)}^\frac{3}{2}} \,&&=\, \mathcal{O}(\beta^{-\frac{1}{2}}), \\
    &\gamma_5 \,&&=\, \frac{\mathcal{O}(\beta^{2\alpha+\frac{1}{2}})}{\V{\chi_a(\Delta_W)}} \,&&=\, \mathcal{O}(\beta^{-\frac{1}{2}}), \qquad &&\gamma_6 \,&&=\, \frac{\mathcal{O}(\beta^{2\alpha})}{\V{\chi_a(\Delta_W)}} \,&&=\, \mathcal{O}(\beta^{-1}).
\end{alignat*}
\end{proof}

\subsection{Increasing observation window} \label{Sec:ObservationWindow}

After the previous subsection examined the asymptotic scenario $\beta \rightarrow \infty$, where the intensity tends to infinity while keeping the observation window $W\in\Bf$ fixed, this subsection addresses the limiting behavior of the Euler characteristic as the measure of the observation window approaches infinity.
To this end, we again fix  $\mathbf{0} \neq a = (a_0, \dots, a_\alpha) \in \mathbb{R}^{\alpha + 1}$ with $a_\alpha \neq 0$ and initially consider a sequence $(W_n)$ in $\Bf$ with $|W_n| \rightarrow \infty$ as $n \rightarrow \infty$.
For a particular class of connection functions that can be termed highly non-integrable, a central limit theorem can be directly formulated using the results from Section \ref{Sec:Intensity}.

\begin{corollary} \label{Kor:ZGWS}
Let $(W_n)$ be a sequence in $\Bf$ with $|W_n| \rightarrow \infty$ as $n \rightarrow \infty$.
Assume there is a $p>0$ such that
\begin{align*}
\varphi_j \geq p \qquad \text{for all } j\in \{1,\dots,\alpha\}.
\end{align*}
Then the convergence
\begin{align*}
\frac{\chi_a\big(\Delta_{W_n}\big)-\E{ \chi_a\big(\Delta_{W_n}\big) }}{\sqrt{\V{ \chi_a\big(\Delta_{W_n}\big) }}} \;\xrightarrow{d}\; \mathcal{N}(0,1) \qquad \text{for } n \rightarrow \infty
\end{align*}
holds with the same rates of convergence of the quantities from Theorem \ref{Th:myRCM} as in Theorem \ref{Th:ZGWS1} (when replacing the intensity $\beta$ with $|W_n|$).
\end{corollary}

\begin{proof}
We first fix an observation window $W \in \Bf$.
From the proof of Lemma \ref{Lem:EulerIntensity}, we can directly conclude that the statement of Lemma \ref{Lem:EulerIntensity} also holds when $\beta$ is replaced by $|W|$.
For a simplicial complex $K$ on $r$ vertices that contains $s$ simplices of dimension at least 1, the condition $\varphi_j \geq p$ implies the inequality $I_K(W) \geq p^s |W|^r$.
Therefore, it follows from Corollary \ref{Kor:Moments_Euler} that the variance $\V{\chi_a(\Delta_W)}$ grows asymptotically at least like a polynomial of degree $2\alpha + 1$ in $|W|$ as $|W|\rightarrow\infty$.
Using these estimates, the desired statement can be derived from Theorem \ref{Th:myRCM} and Corollary \ref{Kor:4thMoment} analogously to the proof of Theorem \ref{Th:ZGWS1}.
\end{proof}

For the remainder of this section, let the integrability condition
\begin{align} \label{Integrabilitaet}
\nu \,:=\, \sup\limits_{x\in\X} \; \int_{\X} \varphi_1(x,y)^\frac{1}{2} \; \lambda (\d y) \,<\, \infty
\end{align}
hold.
The integral $\int_{\X} \varphi_1(x,y) \; \lambda (\d y)$ is, up to the constant $\beta$, the expected edge degree of the vertex $x$ in $\Delta^x$.
The additional exponent $\frac{1}{2}$ is present for technical reasons, which will become clear throughout the section.
Ignoring this, (\ref{Integrabilitaet}) can be interpreted as globally bounded expected edge degrees of added vertices.
Consequently, the situation in Corollary \ref{Kor:ZGWS} is already excluded.
With this integrability condition, we can estimate integral representations of simplicial complexes, which is the subject of the following lemma.

\begin{lemma} \label{Lem:IntForm}
Let $K$ be a simplicial complex on $r$ vertices with $m$ connected components and assume that (\ref{Integrabilitaet}) holds.
Then the integral representation $I_K(W)$ of $K$ over an observation window $W$ satisfies 
\begin{align*}
I_K(W) \,\leq\, \nu^{r-m} |W|^m.
\end{align*}
\end{lemma}

\begin{proof}
First, let $m=1$.
We choose a spanning tree $T$ of the 1-skeleton of $K$ and label the vertices of $K$ as $x_1, \dots, x_r$ such that, for $i \geq 2$, each $x_i$ has a unique neighbor in $T$ within the set $\{x_1, \dots, x_{i-1}\}$.
Such a numbering of the vertices always exists according to Corollary 1.5.2 in \cite{Diestel}.
For $i \geq 2$, let $j_i \in\{1,\dots,i-1\}$ denote the index of the unique neighbor of $x_i$ in $\{x_1, \dots, x_{i-1}\}$.
Then it follows that
\begin{align*} 
I_K(W) \,&\leq\, \int_{W^r} \; \prod_{i=2}^{r} \varphi_1(x_{j_i},x_i) \;\; \lambda^r \, (\d {(x_1,\dots,x_r)}) \nonumber \\
&=\, \int_{W^{r-1}} \; \prod_{i=2}^{r-1} \varphi_1(x_{j_i},x_i) \; \int_W \; \varphi_1(x_{j_r},x_r) \;\; \lambda \, (\d x_r)  \;\; \lambda^{r-1} \, (\d {(x_1,\dots,x_{r-1})}) \nonumber \\
&\leq\, \nu \, \int_{W^{r-1}} \; \prod_{i=2}^{r-1} \varphi_1(x_{j_i},x_i) \;\; \lambda^{r-1} \, (\d {(x_1,\dots,x_{r-1})}) \nonumber \\
&\leq\, \dots \leq \; \nu^{r-2} \, \int_{W^{2}} \; \varphi_1(x_{1},x_2) \;\; \lambda^{2} \, (\d {(x_1,x_{2})}) \;\leq\; \nu^{r-1} |W|.
\end{align*}
For arbitrary $m\in\N$, let $Z_1, \dots, Z_m$ be the $m$ connected components of $K$ and $r_j$ the number of vertices of $Z_j$.
Then it is evident that $\sum_{j=1}^m r_j = r$.
Since the integral representation of a simplicial complex is the product of the integral representations of its connected components, the first part of the proof yields
\begin{align*}
I_K(W) \,=\, I_{Z_1}(W) \cdot\,\dots\,\cdot I_{Z_m}(W) \,\leq\, \prod_{j=1}^m \nu^{r_j-1} |W| \,=\, \nu^{r-m} |W|^m.
\end{align*}
\end{proof}

Lemma \ref{Lem:IntForm} provides the existence of a constant $C>0$ with
\begin{align*}
\V{\chi_a\big(\Delta_W\big)} \,\leq\, C|W| \qquad \text{for all } W\in\Bf,
\end{align*}
since the simplicial complexes associated to the integral representations that appear in Corollary \ref{Kor:Moments_Euler} are all connected.
Furthermore, we can use Lemma \ref{Lem:IntForm} to estimate the quantities from Theorem \ref{Th:myRCM}.

\begin{lemma} \label{Lem:EulerWindow}
Let $F=\chi_a\big(\Delta_W\big)$.
There are constants $C_i>0$, $i\in\{1,\dots,7\}$, not depending on $W\in\Bf$, such that
\begin{align}
\int_{W^3} \E{ \left( \Lambda_{(x_1,U_1),(x_2,U_2)}^1F \right)^2 \left( \Lambda_{(x_2,U_2),(x_1,U_1)}^1F \right)^2}^\frac{1}{2} \hspace{4cm} & \nonumber \\
\times \, \E{ \left(\Lambda_{(x_1,U_1),(x_3,U_3),(x_2,U_2)}^2 F\right)^2 \left(\Lambda_{(x_2,U_2),(x_3,U_3),(x_1,U_1)}^2 F\right)^2}^\frac{1}{2}  \; \lambda^3 \, (\d {x}) &\;\leq\; C_1|W|, \label{E1} \\[8pt]
\int_{W^3} \E{ \left(\Lambda_{(x_1,U_1),(x_3,U_3),(x_2,U_2)}^2 F\right)^2 \left(\Lambda_{(x_2,U_2),(x_3,U_3),(x_1,U_1)}^2 F\right)^2} \; \lambda^3 \, (\d {x}) &\;\leq\; C_2|W|, \label{E2} \\[8pt]
\int_W \E{ |\Lambda_{(x,U)}F|^3 } \; \lambda \, (\d x) &\;\leq\; C_3|W|, \label{E3} \\[8pt]
\int_W \E{ \left(\Lambda_{(x,U)}F\right)^4 }^\frac{3}{4} \; \lambda \, (\d x) &\;\leq\; C_4|W|, \label{E4} \\[8pt]
\int_W \E{ \left(\Lambda_{(x,U)}F\right)^4 } \; \lambda \, (\d x) &\;\leq\; C_5|W|, \label{E5} \\[8pt]
\int_{W^2} 6\; \E{ \left(\Lambda_{(x_1,U_1)}F\right)^4 }^\frac{1}{2} \E{ \left(\Lambda_{(x_1,U_1),(x_2,U_2)}^2 F\right)^4 }^\frac{1}{2} \hspace{3.66cm} & \nonumber \\
+ 3\; \E{ \left(\Lambda_{(x_1,U_1),(x_2,U_2)}^2 F\right)^4 } \; \lambda^2 \, (\d {x}) &\;\leq\; C_6|W|, \label{E6} \\[8pt]
\int_W \E{ \left(\Lambda_{(x,U)}F\right)^4 }^\frac{1}{2} \; \lambda \, (\d x) &\;\leq\; C_7|W|, \label{E7}
\end{align}
where the integration is to be understood with respect to the points $x_i$.
\end{lemma}

\begin{proof}
First, all integrals are finite, which follows from Lemma \ref{Lem:EulerDiffOp} and
\begin{align*}
   f_i^{x_1,\dots,x_k} \big(\Delta_W^{x_1,\dots,x_l}\big) \,\leq\, \binom{\tau}{i+1-k} \,\leq\, \tau^{i+1-k}
\end{align*}
for $k\leq i+1$, where $\tau$ denotes the number of points of $\Phi$ in $W$, which follows a Poisson distribution with parameter $\beta |W|$.
We will begin by demonstrating the detailed calculation for (\ref{E2}) and write, for short, $\Delta_W^x:=\Delta_W^{x_1,x_2,x_3}$.
Using the representation from Lemma \ref{Lem:EulerDiffOp} of the difference operators, we obtain that the left-hand side of (\ref{E2}) is equal to
\begin{align} \label{Darstellung1}
&\sum_{i,j,m,l=1}^\alpha a_ia_ja_ma_l \,\int_{W^3} \E{ f_i^{x_1,x_3}\big(\Delta_W^{x}\big) f_j^{x_1,x_3}\big(\Delta_W^{x}\big) f_m^{x_2,x_3}\big(\Delta_W^{x}\big) f_l^{x_2,x_3}\big(\Delta_W^{x}\big) } \; \lambda^3 \, (\d {x}).
\end{align}
For a set $I$ we write $T^{i,j,m,l}(I)$ for the set of tuples $(\sigma_1,\sigma_2,\sigma_3,\sigma_4)$ of subsets of $I$ satisfying
\begin{align*}
\sigma_1\cup\sigma_2\cup\sigma_3\cup\sigma_4 = I, \enskip |\sigma_1|=i-1, \enskip |\sigma_2|=j-1, \enskip |\sigma_3|=m-1, \enskip |\sigma_4|=l-1.
\end{align*}
Writing $\sigma_k^{x_{p},x_{q}}$ for the simplex consisting of the elements of $\sigma_k$ together with the points $x_{p}$ and $x_{q}$ we see that the integral from (\ref{Darstellung1}) can be written as
Writing $\sigma_k^{x_{p},x_{q}}$ for the simplex consisting of the elements of $\sigma_k$ together with the points $x_{p}$ and $x_{q}$ we see that the expectation in the integral in (\ref{Darstellung1}) can be written as
\begin{align*}
    \mathbb{E}\Bigg[ \sum_{r=\max\{i,j,m,l\}-1}^{i+j+m+l-4} \, \sum_{\substack{I\subseteq V(\Delta_W), \\ |I|=r}} \sum_{\substack{(\sigma_1,\sigma_2,\sigma_3,\sigma_4) \\ \in T^{i,j,m,l}(I)}} \1\big\{ \sigma_1^{x_1,x_3},\sigma_2^{x_1,x_3},\sigma_3^{x_2,x_3},\sigma_4^{x_2,x_3}\in \Delta_W^{x} \big\} \Bigg],
\end{align*}
where $V(\Delta_W)$ denotes the vertex set of $\Delta_W$.
We apply the multivariate Mecke equation (Theorem 4.4 in \cite{Last.Lectures}) to each summand of the sum with respect to $r$.
Together with the notation
\begin{align} \label{Definition_SK}
    K(\sigma_1,\sigma_2,\sigma_3,\sigma_4) \,:=\, \big\langle \sigma_1^{x_1,x_3},\sigma_2^{x_1,x_3},\sigma_3^{x_2,x_3},\sigma_4^{x_2,x_3} \big\rangle,
\end{align}
we obtain that the left-hand side of (\ref{E2}) is equal to
\begin{align*}
    \sum_{i,j,m,l=1}^\alpha a_ia_ja_ma_l \sum_{r=\max\{i,j,m,l\}-1}^{i+j+m+l-4} \, \sum_{(\sigma_1,\sigma_2,\sigma_3,\sigma_4)\in T^{i,j,m,l}(r)} \frac{\beta^r}{r!}\, I_{K(\sigma_1,\sigma_2,\sigma_3,\sigma_4)}(W),
\end{align*}
where all sums are finite and $T^{i,j,m,l}(r):=T^{i,j,m,l}(\{1,\dots,r\})$.
Observe that the complex from (\ref{Definition_SK}) is always connected, since every vertex shares a simplex with $x_3$.
All in all, we see that the left-hand side of (\ref{E2}) is a linear combination of integral representations whose associated simplicial complexes are connected.
Lemma \ref{Lem:IntForm} now provides the estimation from (\ref{E2}).

In principle, this approach can be applied to all integrals from (\ref{E1})-(\ref{E7}).
The only relevant differences are that some integrals involve a product of two expectations, and some expectations have exponents other than 1.
We will demonstrate how the proof works for (\ref{E1}), as this is where both situations occur, and the integral in (\ref{E1}) is the most complicated.
We will keep the steps, which follow analogously to the first part of the proof, as brief as possible.
Using the elementary inequality
\begin{align} \label{elementareUGL}
    (x+y)^p \,\leq\, x^p + y^p, \qquad x,y\geq 0, \enskip p\in (0,1],
\end{align}
we can bound the left-hand side of (\ref{E1}) by
\begin{align}
&\sum_{\substack{i_1,j_1,m_1,l_1=0, \\ i_2,j_2,m_2,l_2=1}}^\alpha \left( \prod_{k=1}^2 \sqrt{\big|a_{i_k}a_{j_k}a_{m_k}a_{l_k}\big|} \right) \, \int_{W^3} \bigg( \E{ f_{i_1}^{x_1}\big(\Delta_W^{x}\big) f_{j_1}^{x_1}\big(\Delta_W^{x}\big) f_{m_1}^{x_2}\big(\Delta_W^{x}\big) f_{l_1}^{x_2}\big(\Delta_W^{x}\big) } \bigg)^\frac{1}{2} \nonumber \\
&\qquad \times\, \bigg( \E{ f_{i_2}^{x_1,x_3}\big(\Delta_W^{x}\big) f_{j_2}^{x_1,x_3}\big(\Delta_W^{x}\big) f_{m_2}^{x_2,x_3}\big(\Delta_W^{x}\big) f_{l_2}^{x_2,x_3}\big(\Delta_W^{x}\big) } \bigg)^\frac{1}{2} \; \lambda^3 \, (\d {x}) \label{Darstellung2}
\end{align}
To proceed similarly to the first part of the proof, we adapt the notation and write $S^{i,j,m,l}(I)$ for the set of tuples $(\sigma_1,\sigma_2,\sigma_3,\sigma_4)$ of subsets of a set $I$ with
\begin{align*}
\sigma_1\cup\sigma_2\cup\sigma_3\cup\sigma_4 = I, \enskip |\sigma_1|=i, \enskip |\sigma_2|=j, \enskip |\sigma_3|=m, \enskip |\sigma_4|=l
\end{align*}
and $S^{i,j,m,l}(s):=S^{i,j,m,l}(\{1,\dots,s\})$.
Furthermore, let
\begin{align*}
    L(\sigma_1,\sigma_2,\sigma_3,\sigma_4) \,:=\, \big\langle \sigma_1^{x_1},\sigma_2^{x_1},\sigma_3^{x_2},\sigma_4^{x_2} \big\rangle.
\end{align*}
With the same argument as in the first part of the proof, we can transform the product of the two expectations (without the exponent $\frac{1}{2}$) that appear in (\ref{Darstellung2}) to
\begin{align*}
    &\sum_{s=\max\{i_1,j_1,m_1,l_1\}}^{i_1+j_1+m_1+l_1} \, \sum_{\substack{(\sigma_1,\sigma_2,\sigma_3,\sigma_4) \\ \in S^{i_1,j_1,m_1,l_1}(s)}} \frac{\beta^s}{s!}\, \int_{W^s} f_{L(\sigma_1,\sigma_2,\sigma_3,\sigma_4)}(x_1,\dots,x_{s+3}) \; \lambda^s \, (\d {(x_4,\dots,x_{s+3})}) \\
    &\times\hspace*{-10pt} \sum_{r=\max\{i_2,j_2,m_2,l_2\}-1}^{i_2+j_2+m_2+l_2-4} \sum_{\substack{(\rho_1,\rho_2,\rho_3,\rho_4) \\ \in T^{i_2,j_2,m_2,l_2}(r)}} \hspace{-8pt} \frac{\beta^r}{r!}\, \int_{W^r} f_{K(\rho_1,\rho_2,\rho_3,\rho_4)}(x_1,\dots,x_{r+3}) \; \lambda^r \, (\d {(x_4,\dots,x_{r+3})}).
\end{align*}
Observe that, when the free vertices appearing in the two simplex functions are taken to be distinct, the product of these functions is itself a simplex function, specifically, the simplex function of the complex
\begin{align} \label{Definition_SK2}
   M(\sigma,\rho) \,:=\, M(\sigma_1,\dots,\sigma_4,\rho_1,\dots,\rho_4) \,:=\, \big\langle \sigma_1^{x_1},\sigma_2^{x_1},\sigma_3^{x_2},\sigma_4^{x_2},\rho_1^{x_1,x_3},\rho_2^{x_1,x_3},\rho_3^{x_2,x_3},\rho_4^{x_2,x_3} \big\rangle
\end{align}
on $r+s+3$ vertices, where $(\sigma_1,\dots,\sigma_4)\in S^{i_1,j_1,m_1,l_1}(I_s)$ and $(\rho_1,\dots,\rho_4)\in T^{i_2,j_2,m_2,l_2}(J_{s,r})$ with
\begin{align*}
I_s:=\{1,\dots,s\}, \qquad J_{s,r}:=\{s+1,\dots,s+r\},
\end{align*}
with the natural convention that $I_0=J_{s,0}=\emptyset$.
Applying the inequality from (\ref{elementareUGL}), we can bound the term from (\ref{Darstellung2}) by
\begin{align}
    &\sum_{\substack{i_1,j_1,m_1,l_1=0, \\ i_2,j_2,m_2,l_2=1}}^\alpha \left( \prod_{k=1}^2 \sqrt{\big|a_{i_k}a_{j_k}a_{m_k}a_{l_k}\big|} \right) \, \sum_{s=\max\{i_1,j_1,m_1,l_1\}}^{i_1+j_1+m_1+l_1} \sum_{r=\max\{i_2,j_2,m_2,l_2\}-1}^{i_2+j_2+m_2+l_2-4} \label{Darstellung3} \\
&\quad\times\sum_{(\sigma_1,\sigma_2,\sigma_3,\sigma_4)\in S^{i_1,j_1,m_1,l_1}(I_s)} \sum_{(\rho_1,\rho_2,\rho_3,\rho_4)\in T^{i_2,j_2,m_2,l_2}(J_{s,r})} \frac{\beta^{s+r}}{s!r!} \nonumber \\
    &\quad\times\, \int_{W^3} \Bigg( \int_{W^{s+r}} f_{M(\sigma,\rho)}(x_1,\dots,x_{s+r+3}) \; \lambda^{s+r} \, (\d {(x_4,\dots,x_{s+r+3})}) \Bigg)^{\frac{1}{2}} \; \lambda^3 \, (\d {(x_1,\dots,x_3)}). \nonumber
\end{align}
We conclude the proof by showing that the integrals appearing there can be bounded by a constant times $|W|$.
Since all sums are finite, the claim then follows.
To this end, we fix a summand and consider the integral appearing there.
The crucial observation is as follows.
The complex in (\ref{Definition_SK2}) is generated by simplices, each of which contains at least one of the points $x_1,x_2,x_3$.
Consequently, each vertex of this complex is connected to at least one of these three points by an edge.
In addition, the edges $[x_1,x_3]$ and $[x_2,x_3]$ are part of the complex.
Hence, there exist $i_j\in\{1,2,3\}$, $j\in\{4,\dots,s+r+3\}$ such that
\begin{align*}
f_{M(\sigma,\rho)}(x_1,\dots,x_{s+r+3}) \;\leq\; \varphi_1(x_1,x_3)\, \varphi_1(x_2,x_3)\; \prod_{j=4}^{s+r+3}\, \varphi_1(x_j,x_{i_j}).
\end{align*}
In particular, $M(\sigma,\rho)$ is connected.
With this, the integral in (\ref{Darstellung3}) can be estimated using the integrability condition (\ref{Integrabilitaet}) in the following manner
\begin{align*}
&\int_{W^3} \Bigg( \int_{W^{s+r}} f_{M(\sigma,\rho)}(x_1,\dots,x_{s+r+3}) \; \lambda^{s+r} \, (\d {(x_4,\dots,x_{s+r+3})}) \Bigg)^\frac{1}{2} \; \lambda^3 \, (\d {(x_1,x_2,x_{3})}) \\
&\qquad \leq\, \int_{W^3} \varphi_1(x_1,x_3)^\frac{1}{2}\,\varphi_1(x_2,x_3)^\frac{1}{2}\; \Bigg( \prod_{j=4}^{s+r+3} \int_{W} \, \varphi_1(x_j,x_{i_j}) \; \lambda \, (\d {x_j}) \Bigg)^\frac{1}{2} \; \lambda^3 \, (\d {(x_1,x_2,x_{3})}) \\
&\qquad \leq\, \int_{W^3} \varphi_1(x_1,x_3)^\frac{1}{2}\varphi_1(x_2,x_3)^\frac{1}{2}\; \nu^{\frac{s+r}{2}} \; \lambda^3 \, (\d {(x_1,x_2,x_{3})}) \\
&\qquad \leq\, \nu^{\frac{s+r}{2}} \int_{W^2} \varphi_1(x_2,x_3)^\frac{1}{2} \int_W \varphi_1(x_1,x_3)^\frac{1}{2}\;\; \lambda \, (\d {x_1}) \; \lambda^2 \, (\d {(x_2,x_{3})}) \\
&\qquad \leq\, \nu^{\frac{s+r}{2}+1} \int_{W} \int_W \varphi_1(x_2,x_3)^\frac{1}{2}\;\; \lambda \, (\d {x_2}) \; \lambda \, (\d {x_3}) \\
&\qquad \leq\, \nu^{\frac{s+r}{2}+2}\, |W|.
\end{align*}
With this method, the estimates (\ref{E1})-(\ref{E7}) can all be established.
It should be noted that for the proof, it is essential that the exponents (outside) of the expectations in the integrals from (\ref{E1})-(\ref{E7}) lie in $[\frac{1}{2},1]$.
For such an exponent $p$, we need $p\leq 1$ in order to apply the inequality (\ref{elementareUGL}).
Conversely, we need $p\geq\frac{1}{2}$ so that in the last chain of inequalities, the individual integrals can always be bounded by $\nu$.
Indeed, the exponent $\frac{1}{2}$ in the integrability condition (\ref{Integrabilitaet}) is due to the fact that the smallest exponent (outside) of the expectations in the integrals from (\ref{E1})-(\ref{E7}) has this value.
Finally, we would like to point out that the case $s=r=0$ in (\ref{Darstellung3}) is possible.
In this case, the complex from (\ref{Definition_SK2}) consists of only the vertices $x_1,x_2,x_3$ and the two edges $[x_1,x_3],[x_2,x_3]$. 
Then the inner integral in (\ref{Darstellung3}) is replaced by the factor $\varphi_1(x_1,x_3)\varphi_1(x_2,x_3)$, so that the further argumentation can still proceed.
\end{proof}

\begin{theorem} [CLT for increasing observation window] \label{Th:ZGWS2}
Let $(W_n)_{n\in\N}$ be a sequence of observation windows with $|W_n|\rightarrow\infty$ for $n\rightarrow\infty$, such that there is a constant $C>0$ with
\begin{align} \label{alphaSimplizes}
\E{ f_\alpha(\Delta_{W_n}) } \,\geq\, C |W_n| \qquad \text{for all sufficiently large } n\in\N.
\end{align}
Moreover, assume that the condition (\ref{Integrabilitaet}) holds.
Then for $a=(a_0,\dots,a_\alpha) \in \R^{\alpha+1}$ with $a_\alpha\neq 0$ it is true that
\begin{align*}
\frac{\chi_a\big(\Delta_{W_n}\big)-\E{ \chi_a\big(\Delta_{W_n}\big) }}{\sqrt{\V{\chi_a\big(\Delta_{W_n}\big) }}} \xrightarrow{d} \mathcal{N}(0,1) \qquad \text{for } n \rightarrow \infty.
\end{align*}
Concretely, the quantities from Theorem \ref{Th:myRCM} applied to $\tilde{\chi}_a\big(\Delta_{W_n}\big)$ satisfy
\begin{align*}
\gamma_1,\gamma_2,\gamma_3,\gamma_4,\gamma_5,\gamma_6 \,=\, \mathcal{O}\left(|W_n|^{-\frac{1}{2}}\right).
\end{align*}
\end{theorem}

\begin{proof}
We apply Theorem \ref{Th:myRCM} to the standardized functional $\widetilde{F}_n$ with $F_n:=\chi_a\big(\Delta_{W_n}\big)$.
Due to Lemma \ref{Lem:EulerWindow} the integrability conditions of Theorem \ref{Th:myRCM} are fulfilled.
Corollary \ref{Kor:4thMoment} and Lemma \ref{Lem:EulerWindow} show the existence of a constant $C_8>0$ with
\begin{align*}
\E{ (F_n-\E{F_n})^4 } \,\leq\, C_8 |W_n|^2 \qquad \text{for all sufficiently large } n\in\N.
\end{align*}
Here, the estimate (\ref{E7}) is used.
Moreover, Proposition \ref{Prop:Cov} and (\ref{lowerBoundEulerVar}) together with the condition (\ref{alphaSimplizes}) imply
\begin{align*}
\V{ \chi_a\big(\Delta_{W_n}\big) } \,\geq\, a_\alpha^2 \, \E{ f_\alpha(\Delta_{W_n}) } \,\geq\, a_\alpha^2 C |W_n| \,=:\, C_9 |W_n| \quad \text{for all sufficiently large } n\in\N.
\end{align*}
We denote by $\vartheta_1^{(n)},\dots,\vartheta_6^{(n)}$ the integrals from (\ref{E1})-(\ref{E6}) (in this order) for the observation window $W_n$ and by $\gamma_1^{(n)},\dots,\gamma_6^{(n)}$ the quantities from Theorem \ref{Th:myRCM} applied to $\widetilde{F}_n$.
Using the estimates from Lemma \ref{Lem:EulerWindow}, we obtain for all sufficiently large $n\in\N$
\begin{align*}
\gamma_1^{(n)} &\,=\, \frac{2\sqrt{\beta^3\vartheta_1^{(n)}}}{\V{\chi_a\big(\Delta_{W_n}\big) }} \,\leq\, \frac{2\sqrt{\beta^3C_1|W_n|}}{C_9|W_n|} \,=\, \frac{2\sqrt{\beta^3C_1}}{C_9} \, |W_n|^{-\frac{1}{2}}, \\[10pt]
\gamma_2^{(n)} &\,=\, \frac{\sqrt{\beta^3\vartheta_2^{(n)}}}{\V{\chi_a\big(\Delta_{W_n}\big) }} \,\leq\, \frac{\sqrt{\beta^3C_2|W_n|}}{C_9|W_n|} \,=\, \frac{\sqrt{\beta^3C_2}}{C_9} \, |W_n|^{-\frac{1}{2}}, \\[10pt]
\gamma_3^{(n)} &\,=\, \frac{\beta\vartheta_3^{(n)}}{\V{\chi_a\big(\Delta_{W_n}\big) }^\frac{3}{2}} \,\leq\, \frac{\beta C_3|W_n|}{\left(C_9|W_n|\right)^\frac{3}{2}} \,=\, \frac{\beta C_3}{C_9^\frac{3}{2}} \, |W_n|^{-\frac{1}{2}}, \\[10pt]
\gamma_4^{(n)} &\,=\, \frac{\beta}{2} \E{ (F_n-\E{F_n})^4 }^\frac{1}{4} \frac{\vartheta_4^{(n)}}{\V{\chi_a\big(\Delta_{W_n}\big) }^2} \,\leq\, \frac{\beta C_8^\frac{1}{4} C_4|W_n|^\frac{3}{2}}{2C_9^2|W_n|^2} \,=\, \frac{\beta C_8^\frac{1}{4} C_4}{2C_9^2} \, |W_n|^{-\frac{1}{2}}, \\[10pt]
\gamma_5^{(n)} &\,=\, \frac{\sqrt{\beta\vartheta_5^{(n)}}}{\V{\chi_a\big(\Delta_{W_n}\big) }} \,\leq\, \frac{\sqrt{\beta C_5|W_n|}}{C_9|W_n|} \,=\, \frac{\sqrt{\beta C_5}}{C_9} \, |W_n|^{-\frac{1}{2}}, \\[10pt]
\gamma_6^{(n)} &\,=\, \frac{\beta\sqrt{\vartheta_6^{(n)}}}{\V{\chi_a\big(\Delta_{W_n}\big) }} \,\leq\, \frac{\beta\sqrt{ C_6|W_n|}}{C_9|W_n|} \,=\, \frac{\beta\sqrt{C_6}}{C_9} \, |W_n|^{-\frac{1}{2}},
\end{align*}
with the constants $C_1,\dots,C_6$ from Lemma \ref{Lem:EulerWindow}.
\end{proof}

\section{Multivariate central limit theorem for simplex counts in the marked stationary case} \label{Sec:MultivariateCLT}

We now turn our attention to an important special case, which we refer to as the marked stationary model, and which has already gained substantial recognition in the literature on the RCM as a random graph.
Let $(\AA, \mathcal{T})$ be a Borel space equipped with a probability measure $\Theta$.
We consider the space $\X = \R^d \times \AA$ with the measure $\lambda_d \otimes \Theta$, where $\lambda_d$ denotes the $d$-dimensional Lebesgue measure on $\R^d$.
Consequently, $\Phi$ forms a Poisson process with intensity measure $\beta\lambda_d \otimes \Theta$.
In this setting, the vertices of $\Delta$ are Euclidean points, each associated with an additional mark in $\AA$, referred to as the mark of the point.
Beyond measurability and symmetry, we assume that the connection functions in the marked stationary model are translation-invariant, that is, that the equality
\begin{align} \label{translation_invariant}
    \varphi_j\big((x_0+t,a_0),\dots,(x_j+t,a_j)\big) \,&=\, \varphi_j\big((x_0,a_0),\dots,(x_j,a_j)\big)
\end{align}
holds for all $j\in\{1,\dots,\alpha\}$, $x_0,\dots,x_j,t\in\R^d$ and $a_0,\dots,a_j\in\AA$.
Note that the unmarked stationary case arises when the mark space $\AA$ consists of a single element.
Central limit theorems for the unmarked stationary case, applicable to a broad class of functionals, can be found in \cite{Can,Last.RCM}.
The general marked case in the context of percolation problems is studied in \cite{Caicedo,Dickson}.
To derive a multivariate central limit theorem for simplex counts in this framework, we begin by analyzing the asymptotic behavior of integral representations.
In order to do so, we consider sequences in $\mathcal{K}^d:=\{\emptyset\neq K\subset\R^d \mid K\text{ is compact and convex}\}$ and denote by
\begin{align*}
r(K) \,:=\, \sup \big\{ r>0 \,\|\, \exists x\in K \text{ s.t. } B(x,r)\subseteq K \big\}
\end{align*}
the inradius of a set $K\in\mathcal{K}^d$.
Furthermore, let 
\begin{align*}
d(x,K) \,:=\, \min \{ \Vert x-y\Vert \,\|\, y\in K \}
\end{align*}
be the distance of a point $x\in\R^d$ to a set $K\in\mathcal{K}^d$ and $\partial K$ the topological boundary of $K$.
Furthermore, we denote by $B(x,r)$ the ball with center $x\in\R^d$ and radius $r>0$ and in this section we write $|W|$ for the $d$-dimensional Lebesgue measure of a measurable set $W\subseteq\R^d$.
Analogously to (\ref{translation_invariant}), we call a function $f:\left(\R^d\right)^m\rightarrow [0,\infty)$ translation-invariant if
\begin{align*}
    f(x_1+t,\dots,x_m+t) \,=\, f(x_1,\dots,x_m) \qquad\text{for all } t,x_1,\dots,x_m\in\R^d.
\end{align*}

\begin{lemma} \label{Lem:translationsinvarianteFkt}
Let $(W_n)_{n\in\N}$ be a sequence in $\K^d$ with $r(W_n)\rightarrow\infty$ for $n\rightarrow\infty$ and $f:\left(\R^d\right)^m\rightarrow [0,\infty)$, $m\in\N$, be a measurable and translation-invariant function such that
\begin{align*}
\nu_f \,:=\, \int_{\left(\R^d\right)^{m-1}} f(\mathbf{0},x_2,\dots,x_{m}) \; \d {(x_2,\dots,x_m)} \,<\, \infty.
\end{align*}
For $m=1$, we read this term as $f(\0)$.
Then it is true that
\begin{align*}
\frac{1}{|W_n|} \int_{W_n^m} f(x_1,\dots,x_{m}) \; \d {(x_1,\dots,x_m)} \;\xrightarrow{\;\;}\; \nu_f \qquad \text{for } n\rightarrow\infty.
\end{align*}
\end{lemma}

\begin{proof}
In the case $m=1$ the translation invariance of $f$ implies that $f$ is constant and the statement is trivial.
Now let $m\geq 2$.
As a consequence of the translation invariance of $f$, we have
\begin{align*}
\nu_f \,=\, \frac{1}{|W_n|} \int_{W_n} \int_{\left(\R^d\right)^{m-1}} f(x_1,x_2,\dots,x_{m}) \; \d {(x_2,\dots,x_m)} \; \d {x_1},
\end{align*}
for each $n\in\N$, since the inner integral does not depend on $x_1$.
Hence, for any $\varepsilon>0$, it is true that
\begin{align} 
&\bigg|\, \nu_f - \frac{1}{|W_n|} \int_{W_n^m} f(x_1,\dots,x_{m}) \; \d {(x_1,\dots,x_m)} \,\bigg| \nonumber \\
& \qquad \qquad =\, \frac{1}{|W_n|} \int_{W_n} \int_{\left(W_n^{m-1}\right)^\mathrm{c}} f(x_1,x_2,\dots,x_{m}) \; \d {(x_2,\dots,x_m)} \; \d {x_1} \nonumber \\
& \qquad \qquad \leq\, \frac{1}{|W_n|} \int_{W_n} \1\big\{ d(x_1,\partial W_n)\geq \varepsilon \big\} \; \d {x_1} \int_{\left(B(\mathbf{0},\varepsilon)^{m-1}\right)^\mathrm{c}} f(\mathbf{0},x_2,\dots,x_{m}) \; \d {(x_2,\dots,x_m)} \nonumber \\
& \qquad \qquad \quad \; + \frac{\nu_f}{|W_n|} \int_{W_n} \1\big\{ d(x_1,\partial W_n)\leq \varepsilon \big\} \; \d {x_1}. \label{Konvergenz2}
\end{align}
We bound the first integral in the first summand from above by $|W_n|$, thereby obtaining an upper bound for the first summand that is independent of $n$.
Together with the Steiner formula (see, for example, (14.5) in \cite{Schneider}) it follows from (3.19) in \cite{Hug} and Lemma 3.7 in \cite{Hug} that 
\begin{align} \label{Randverschwindet}
\frac{|\{x\in W_n \,|\, d(x,\partial W_n)\leq \varepsilon\}|}{|W_n|} \,\leq\, \sum_{j=0}^{d-1} \frac{\varepsilon^{d-j}(2^d-1)}{r(W_n)^{d-j}} \;\xrightarrow{\;\;}\; 0 \qquad \text{for } n\rightarrow\infty.
\end{align}
Therefore, the second summand in (\ref{Konvergenz2}) converges to 0 for $n\rightarrow\infty$.
The dominated convergence theorem implies that the second integral in the first summand of (\ref{Konvergenz2}) vanishes for $\varepsilon\rightarrow\infty$.
\end{proof}

An analogous argument is used in the proof of Lemma 9.2 in \cite{Last.RCM}, where the connection function of the RCM takes on the role of $f$.
The crucial application of Lemma \ref{Lem:translationsinvarianteFkt}, for which we require the integrability condition
\begin{align} \label{Integrabilitaet:esssup}
\esssup\limits_{a\in\AA} \; \int_{\R^d}\int_\AA \varphi_1\big((\0,a),(y,b)\big) \; \Theta(\d b) \; \d y \,<\, \infty,
\end{align}
to hold, yields the following corollary.

\begin{corollary} \label{Kor:Konvergenz_Integralformen}
Let $K$ be a connected simplicial complex on $r\in\N$ vertices and assume that (\ref{Integrabilitaet:esssup}) holds.
Then, for any sequence $(W_n)_{n\in\N}$ in $\K^d$ with $r(W_n)\rightarrow\infty$, we have
\begin{align*}
&\frac{1}{|W_n|} \int_{W_n^r} \int_{\AA^r} f_K\big((x_1,a_1),\dots,(x_r,a_r)\big) \; \Theta^r (\d {(a_1,\dots,a_r)}) \; \d {(x_1,\dots,x_r)} \\
&\qquad \xrightarrow{\;\;}\; \int_{(\R^d)^{r-1}} \int_{\AA^r} f_K\big((\0,a_1),\dots,(x_r,a_r)\big) \; \Theta^r (\d {(a_1,\dots,a_r)}) \; \d {(x_2,\dots,x_r)} \,<\, \infty
\end{align*}
for $n\rightarrow\infty$.
\end{corollary}

\begin{proof}
We define a translation-invariant function $h:\left(\R^d\right)^r\rightarrow [0,\infty)$ by
\begin{align*}
h(x_1,\dots,x_r) \,:=\, \int_{\AA^r} f_K\big((x_1,a_1),\dots,(x_r,a_r)\big) \; \Theta^r (\d {(a_1,\dots,a_r)}).
\end{align*}
We show that the estimate
\begin{align} \label{Abschaetzung_Kor:translationsinvarianteFkt}
\int_{\left(\R^d\right)^{r-1}} h(\mathbf{0},x_2,\dots,x_{r}) \; \d {(x_2,\dots,x_r)} \,\leq\, \rho^{r-1} \,<\, \infty
\end{align}
holds, where $\rho$ denotes the finite value of the essential supremum from (\ref{Integrabilitaet:esssup}).
To this end, we choose a spanning tree $T$ of the 1-skeleton of $K$ and label the vertices of $K$ as $x_1,\dots,x_r$, such that for $i\geq 2$, $x_i$ has a unique neighbor in $T$ in the set $\{x_1,\dots,x_{i-1}\}$, which is always possible according to Corollary 1.5.2 in \cite{Diestel}.
Then, the integral in (\ref{Abschaetzung_Kor:translationsinvarianteFkt}) can be bounded from above by \begin{align}\label{Abschaetzung2_Kor:translationsinvarianteFkt}
\int_{\left(\R^d\right)^{r-1}} \int_{\AA^r} \prod_{i=2}^r \varphi_1\big( (x_{j_i},a_{j_i}),(x_i,a_i) \big) \; \Theta^r (\d {(a_1,\dots,a_r)}) \; \d {(x_2,\dots,x_r)},
\end{align}
where, without loss of generality, we set $x_1=\0$ (due to the translation invariance of the connection functions, it is irrelevant which component of the function $h$ is not integrated out in the integral from (\ref{Abschaetzung_Kor:translationsinvarianteFkt})).
By successively integrating over $(x_i,a_i)$ for $i\in\{2,\dots,r\}$ in descending order of $i$, we use at the $i$-th step the estimate
\begin{align}\label{Abschaetzung3_Kor:translationsinvarianteFkt}
\int_{\R^d}\int_\AA \varphi_1\big( (x_{j_i},a_{j_i}),(x_i,a_i) \big) \; \Theta (\d {a_i}) \; \d {x_i} \,\leq\, \rho,
\end{align}
which holds for $\Theta$-almost all $a_{j_i}$ according to (\ref{Integrabilitaet:esssup}).
Due to the translation invariance of the connection functions, this estimate holds for any point $x_{j_i}\in\R^d$.
After applying the estimate from (\ref{Abschaetzung3_Kor:translationsinvarianteFkt}) $r-1$ times to the integral in (\ref{Abschaetzung2_Kor:translationsinvarianteFkt}), we obtain the upper bound $\rho^{r-1}\int_\AA \Theta(\d{a_1})=\rho^{r-1}$, which proves (\ref{Abschaetzung_Kor:translationsinvarianteFkt}).
The statement now follows directly from Lemma \ref{Lem:translationsinvarianteFkt}.
\end{proof}

We apply Corollary \ref{Kor:Konvergenz_Integralformen} to the complexes from (\ref{Def:Komplex_Kmlr}) for $m,l\in\{1,\dots,\alpha+1\}$ and $\max\{m,l\}\leq r\leq m+l-1$.
Note that the complex $K_{m,l}^r$ is connected in this case.
Assume that the integrability condition (\ref{Integrabilitaet:esssup}) holds.
Writing $\z m l r (\0)$ for the integral
\begin{align*}
\int_{\left(\R^d\right)^{r-1}} \int_{\AA^r} \kappa_{m,l}^r\big((\0,a_1),(x_2,a_2),\dots,(x_r,a_r)\big) \; \Theta^r (\d{(a_1,\dots,a_r)})\; \d {(x_2,\dots,x_r)},
\end{align*}
Corollary \ref{Kor:Konvergenz_Integralformen} implies for each sequence $(W_n)_{n\in\N}$ in $\K^d$ with $r(W_n)\rightarrow\infty$
\begin{align} \label{asymptoticIntform}
\frac{\z m l r (W_n)}{|W_n|} \;\xrightarrow{\enskip}\; \z m l r (\0) \,<\, \infty \qquad \text{for } n \rightarrow \infty.
\end{align}
To avoid confusion with the notation for the marks, we denote the coefficient vector of the generalized Euler characteristic by $b=(b_0,\dots,b_\alpha)$ throughout this section.
Corollary \ref{Kor:Moments_Euler} shows together with (\ref{asymptoticIntform}) for each coefficient vector $b=(b_0,\dots,b_\alpha)\in\R^{\alpha+1}$
\begin{align}
\frac{\E{ \chi_b\big(\Delta_{W_n}\big) }}{|W_n|} \; &\xrightarrow{\enskip}\; \sum_{j=0}^\alpha b_j \frac{\beta^{j+1}}{(j+1)!} \, \z {j+1}{j+1}{j+1} (\0), \nonumber \\
& \hspace{-2.4cm} \frac{ \Cov{ f_{m-1}\big(\Delta_{W_n}\big)}{ f_{l-1}\big(\Delta_{W_n}\big)} } {|W_n|} \; \xrightarrow{\enskip}\; \sum_{r=\max\{m,l\}}^{m+l-1} \frac{\beta^r}{(r-m)!(r-l)!(m+l-r)!} \, \z m l r (\0), \nonumber \\
\frac{\V{ \chi_b\big(\Delta_{W_n}\big) }}{|W_n|} \; &\xrightarrow{\enskip}\; \sum_{m,l=1}^{\alpha+1} b_{m-1} b_{l-1} \sigma_{m,l} \label{asymptoticVar}
\end{align}
for $n\rightarrow\infty$ and $m,l\in\{1,\dots,\alpha+1\}$, where $\sigma_{m,l}$ denotes the limit of the normalized covariance in the second convergence.

\begin{theorem} [Multivariate CLT for simplex counts] \label{Th:ZGWSmulti}
Let $(W_n)_{n\in\N}$ be a sequence in $\K^d$ with $r(W_n)\rightarrow\infty$ for $n\rightarrow\infty$.
Also, assume that $\z {\alpha+1}{\alpha+1}{\alpha+1} (\0) > 0$ and that the integrability condition (\ref{Integrabilitaet}) holds.
Furthermore, let $N_\Sigma$ denote an $(\alpha+1)$-dimensional centered normally distributed random vector with covariance matrix $\Sigma:=(\sigma_{m,l})_{1\leq m,l\leq\alpha+1}$, whose entries $\sigma_{m,l}$, $m,l\in\{1,\dots,\alpha+1\}$, are defined above.
Then $\Sigma$ is a symmetric, positive definite matrix, and we have
\begin{align*}
\frac{1}{\sqrt{|W_n|}} \left( f_0\big(\Delta_{W_n}\big)-\E{f_0\big(\Delta_{W_n}\big)},\dots,f_\alpha\big(\Delta_{W_n}\big)-\E{f_\alpha\big(\Delta_{W_n}\big)} \right) \;\xrightarrow{d}\; N_\Sigma \qquad \text{for } n \rightarrow \infty.
\end{align*}
\end{theorem}

\begin{proof}
First, note that (\ref{Integrabilitaet}) implies condition (\ref{Integrabilitaet:esssup}).
The Cram\'er--Wold theorem (see, for example, Corollary 6.5 in \cite{Kallenberg}) provides the desired convergence statement if, for all coefficient vectors $b\in\R^{\alpha+1}$, the convergence
\begin{align} \label{Verteilungskonvergenz_Linearkombination}
\frac{\chi_b\big(\Delta_{W_n}\big) - \E{\chi_b\big(\Delta_{W_n}\big)}}{\sqrt{|W_n|}} \;\xrightarrow{d}\; \SKP {N_\Sigma}b
\end{align}
holds, since $\chi_b\big(\Delta_{W_n}\big)$ is exactly the linear combination of simplex counts with coefficient vector $b$.
While this statement is trivial for $b=\0$, Theorem \ref{Th:ZGWS2} provides the asymptotic normality of the standardized Euler characteristic $\tilde{\chi}_b\big(\Delta_{W_n}\big)$ for $\0\neq b=(b_0,\dots,b_\alpha)\in \R^{\alpha+1}$ as $n\rightarrow\infty$, if there exists a constant $C>0$ with
\begin{align}\label{lastLabel}
\E{ f_s(\Delta_{W_n}) } \,\geq\, C |W_n| \qquad \text{for all sufficiently large } n\in\N
\end{align}
for $s:=\max\{ i\in\{0,\dots,\alpha\} \,\|\, b_i\neq 0 \}$.
Before proving (\ref{lastLabel}), we proceed with the further argumentation.
According to (\ref{asymptoticVar}), we have
\begin{align}\label{verylastLabel}
\frac{\V{ \chi_b\big(\Delta_{W_n}\big) }}{|W_n|} \;\xrightarrow{\enskip}\; \sum_{m,l=1}^{\alpha+1} b_{m-1}b_{l-1} \sigma_{m,l} \,=\, \V{ \langle N_{\Sigma},b \rangle }
\end{align}
and using Slutsky's lemma, the distributional convergence (\ref{Verteilungskonvergenz_Linearkombination}) follows.
According to Proposition \ref{Prop:Cov} and (\ref{lowerBoundEulerVar}), $\V{ \chi_b\big(\Delta_{W_n}\big) }\geq b_s^2\E{ f_s(\Delta_{W_n}) }$ holds.
Thus, (\ref{lastLabel}) implies that the limit in (\ref{verylastLabel}) is positive, from which the positive definiteness of $\Sigma$ follows.

We complete the proof by showing that (\ref{lastLabel}) holds for all $s\in\{0,\dots,\alpha\}$.
Proposition \ref{Prop:Cov} (i) yields $\E{ f_\alpha(\Delta_{W_n}) } = \frac{\beta^{\alpha+1}}{(\alpha+1)!} \z {\alpha+1}{\alpha+1}{\alpha+1} (W_n)$.
Together with (\ref{asymptoticIntform}), it follows that
\begin{align*}
\frac{\E{ f_\alpha(\Delta_{W_n}) }}{|W_n|} \;\xrightarrow{\enskip}\; \frac{\beta^{\alpha+1}}{(\alpha+1)!} \z {\alpha+1}{\alpha+1}{\alpha+1} (\0) \,>\, 0 \qquad \text{for } n \rightarrow \infty,
\end{align*}
from which (\ref{lastLabel}) follows for $s=\alpha$.
When considering independent marks $V\sim\Theta$ and $U\sim\Q$, which are also independent of $\Psi$, for the simplicial complex $\Delta^\0:=T(\Psi+\delta_{(\0,V,U)})$, Mecke's formula implies that $\frac{\beta^s}{s!}\z {s+1}{s+1}{s+1} (\0)$ represents the expected number of $s$-simplices in $\Delta^\0$ that contain the added point $(\0,V)$.
Thus, $\z {\alpha+1}{\alpha+1}{\alpha+1} (\0)>0$ also guarantees that $\z {s+1}{s+1}{s+1} (\0)>0$ for all $s\in\{0,\dots,\alpha-1\}$.
This leads to (\ref{lastLabel}) for arbitrary $s\in\{0,\dots,\alpha-1\}$ analogously to the argument for $s=\alpha$.
\end{proof}

\section*{Acknowledgements}

The results of this paper stem from the author's PhD thesis \cite{Pabst.Thesis}.
The author is grateful to Daniel Hug, the supervisor of the PhD thesis, for the support and insightful feedback throughout the research.
This work was partially supported by the Deutsche Forschungsgemeinschaft (DFG, German Research Foundation) through the SPP 2265, under grant numbers HU 1874/5-1 and ME 1361/16-1.

\bigskip
\bigskip
\bigskip
\bigskip
\bigskip

\bibliographystyle{abbrv}
\bibliography{RandomConnectionModel_final}

\begin{thebibliography}{10}

\bibitem{Caicedo}
A.~Caicedo and M.~Dickson.
\newblock Critical exponents for marked random connection models.
\newblock {\em Electronic Journal of Probability}, 29(151):1--57, 2024.

\bibitem{Can}
V.~H. Can and K.~D. Trinh.
\newblock Random connection models in the thermodynamic regime: central limit
  theorems for add-one cost stabilizing functionals.
\newblock {\em Electronic Journal of Probability}, 27(36):1--40, 2022.

\bibitem{Candela}
J.~D. Candela.
\newblock Central limit theorems for soft random simplicial complexes.
\newblock \href{https://arxiv.org/abs/2311.10625}{arXiv:2311.1065}, 2025.

\bibitem{Costa}
A.~Costa and M.~Farber.
\newblock Random simplicial complexes.
\newblock In F.~Callegaro, F.~Cohen, C.~De~Concini, E.~M. Feichtner, G.~Gaiffi,
  and M.~Salvetti, editors, {\em Configuration spaces: geometry, topology and
  representation theory}, volume~14, pages 129--153. Springer, Cham, 2016.

\bibitem{Dickson}
M.~Dickson and M.~Heydenreich.
\newblock The triangle condition for the marked random connection model.
\newblock \href{https://arxiv.org/abs/2210.07727}{arXiv:2210.07727}, 2022.

\bibitem{Diestel}
R.~Diestel.
\newblock {\em Graph theory}, volume 173 of {\em Graduate Texts in
  Mathematics}.
\newblock Springer, Berlin, Heidelberg, fifth edition, 2017.

\bibitem{Hirsch}
C.~Hirsch and D.~Valesin.
\newblock Face and cycle percolation.
\newblock {\em Journal of Applied and Computational Topology}, 9(6):1--42,
  2025.

\bibitem{Hug}
D.~Hug, G.~Last, and M.~Schulte.
\newblock Second-order properties and central limit theorems for geometric
  functionals of boolean models.
\newblock {\em The Annals of Applied Probability}, 26(1):73--135, 2016.

\bibitem{Kahle.Topology}
M.~Kahle.
\newblock Topology of random clique complexes.
\newblock {\em Discrete Mathematics}, 309(6):1658--1671, 2009.

\bibitem{Kahle.LimitTheorems}
M.~Kahle and E.~Meckes.
\newblock Limit theorems for betti numbers of random simplicial complexes.
\newblock {\em Homology, Homotopy and Applications}, 15(1):343--374, 2013.

\bibitem{Kallenberg}
O.~Kallenberg.
\newblock {\em Foundations of modern probability}, volume~99 of {\em
  Probability Theory and Stochastic Modelling}.
\newblock Springer, Cham, third edition, 2021.

\bibitem{Last.RCM}
G.~Last, F.~Nestmann, and M.~Schulte.
\newblock The random connection model and functions of edge-marked poisson
  processes: second order properties and normal approximation.
\newblock {\em The Annals of Applied Probability}, 31(1):128--168, 2021.

\bibitem{Last.NormalApprox}
G.~Last, G.~Peccati, and M.~Schulte.
\newblock Normal approximation on poisson spaces: Mehler's formula, second
  order poincar{\'e} inequalities and stabilization.
\newblock {\em Probability Theory and Related Fields}, 165:667--723, 2016.

\bibitem{Last.Lectures}
G.~Last and M.~Penrose.
\newblock {\em Lectures on the Poisson process}.
\newblock Cambridge University Press, Cambridge, 2018.

\bibitem{Pabst.Thesis}
D.~Pabst.
\newblock {\em Das Random Connection Model für höherdimensionale
  Simplizialkomplexe}.
\newblock PhD thesis, Karlsruhe Institute of Technology, 2024.

\bibitem{Penrose.RCM}
M.~Penrose.
\newblock On a continuum percolation model.
\newblock {\em Advances in Applied Probability}, 23(3):536--556, 1991.

\bibitem{Penrose.Graphs}
M.~Penrose.
\newblock {\em Random geometric graphs}.
\newblock Oxford studies in probability. Oxford University Press, Oxford, 2003.

\bibitem{Ratcliffe}
J.~G. Ratcliffe.
\newblock {\em Foundations of hyperbolic manifolds}, volume 149 of {\em
  Graduate Texts in Mathematics}.
\newblock Springer, Cham, 3rd edition, 2019.

\bibitem{Schneider}
R.~Schneider and W.~Weil.
\newblock {\em Stochastic and integral geometry}.
\newblock Probability and Its Applications. Springer, Berlin, Heidelberg, 2008.

\bibitem{Thomas}
A.~M. Thomas and T.~Owada.
\newblock Functional limit theorems for the euler characteristic process in the
  critical regime.
\newblock {\em Advances in Applied Probability}, 53:57--80, 2021.

\bibitem{Yogeshwaran}
D.~Yogeshwaran, E.~Subag, and R.~J. Adler.
\newblock Random geometric complexes in the thermodynamic regime.
\newblock {\em Probability Theory and Related Fields}, 167:107--142, 2017.

\end{thebibliography}

\end{document}